\documentclass[reqno]{amsart}

\numberwithin{equation}{section} % Equation numbering control.
\usepackage{amsmath,amsfonts,amssymb,amsthm,latexsym}
\usepackage{enumerate}
\usepackage{cancel}
\usepackage{stackrel}%\usepackage{ulem}

\usepackage{cases}
\usepackage[square,comma,numbers,sort&compress]{natbib}

\usepackage[colorlinks=true]{hyperref}
\hypersetup{urlcolor=blue, citecolor=red}

\newcounter{mnote}
\theoremstyle{plain}
\newtheorem{theorem}{Theorem}[section]
\newtheorem{proposition}[theorem]{Proposition}
\newtheorem{lemma}[theorem]{Lemma}

\theoremstyle{definition}

\theoremstyle{remark}
\newtheorem{remark}[theorem]{Remark}

\newcommand{\field}[1]{\mathbb{#1}}
\newcommand{\nN}{\field{N}}

\newcommand{\nR}{\field{R}}

\newcommand{\vphi}{\varphi}

\newcommand{\pd}[2]{\frac{\partial #1}{\partial #2}}
\newcommand{\od}[2]{\frac{d #1}{d #2}}

\newcommand{\abs}[1]{\left\lvert#1\right\rvert}
\newcommand{\norm}[1]{\left\lVert#1\right\rVert}

\newcommand{\Lph}[1]{\text{$L^{#1}$}(\Om)}
\newcommand{\Hph}[1]{\text{$H^{#1}$}(\Om)}
\newcommand{\Wph}[2]{\text{$W^{#1,#2}$}(\Om)}

\newcommand{\Hphd}[1]{\text{$H^{#1}_0$}(\Om)}
\newcommand{\Hphp}[1]{\text{${\dot{H}}^{#1}$}(\Om)}

\newcommand{\dU}{\tilde{u}}

\newcommand{\dP}{\tilde{p}}
\newcommand{\uu}{U}

\newcommand{\p}{P}

\def\Dd{\Delta}

\def\Om{\Omega}

\def\pp{\partial}

\begin{document}
 % ---------------------- Article Information ----------------------
\title[abridged continuous data assimilation ]{Abridged continuous data assimilation for the 2D Navier-Stokes equations utilizing measurements of only one component of the velocity field}

\date{\today}

% ----------------------  Author Information ----------------------
%
\author{Aseel Farhat}
\address[Aseel Farhat]{Department of Mathematics\\
                Indiana University, Bloomington\\
        Bloomington, IN 47405, USA}
\email[Aseel Farhat]{afarhat@indiana.edu}
\author{Evelyn Lunasin}
\address[Evelyn Lunasin]{Department of Mathematics\\
              United States Naval Academy\\
      Annapolis, MD 21401, USA}
\email[Evelyn M. Lunasin]{lunasin@usna.edu}

\author{Edriss S. Titi}
\address[Edriss S. Titi]{Department of Mathematics, Texas A\&M University, 3368 TAMU,
 College Station, TX 77843-3368, USA.  {\bf ALSO},
  Department of Computer Science and Applied Mathematics, Weizmann Institute
  of Science, Rehovot 76100, Israel.} \email{titi@math.tamu.edu and
  edriss.titi@weizmann.ac.il}

%---------------------------------------------------------------
\begin{abstract}

We introduce a continuous data assimilation (downscaling) algorithm for the two-dimensional Navier-Stokes equations employing coarse mesh measurements of only one component of the velocity field. This algorithm can be implemented with a variety of finitely many observables: low Fourier modes, nodal values, finite volume averages, or finite elements.  We provide conditions on the spatial resolution of the observed data, under the assumption that the observed data is free of noise, which are sufficient to show that the solution of the algorithm approaches, at an exponential rate asymptotically in time, to the unique exact unknown reference solution, of the 2D Navier-Stokes equations, associated with the observed (finite dimensional projection of) velocity.

\end{abstract}

 \maketitle

%--------------------------------------------------------------------
%\tableofcontents

\section{Introduction}\label{intro}
Data assimilation is a downscaling process for estimating the state of a system by synchronizing information from collected coarse mesh measured data and prediction derived from numerical forecast model. The classical method of continuous data assimilation, see, e.g., \cite{Daley}, is
to insert observational measurements directly into a model as the
latter is being integrated in time. One way to exploit this is to insert Fourier low mode observables from a time series into the equation for the evolution of the high modes. After a relatively short time $t=\bar{t}$, the solution to the equation for the high modes is close to the high modes of the exact reference solution associated with the observables. At that point the low modes and high modes can be combined to form a complete good approximation of the state of the system at time $t = \bar{t}$, which can then be used as an initial condition for a high resolution simulation. This was the approach taken for the 2D Navier-Stokes in \cite{Browning_H_K, B_L_Stuart, Henshaw_Kreiss_Ystrom,Hayden_Olson_Titi,Olson_Titi_2003, Olson_Titi_2008, Korn}. Except of the work in \cite{B_L_Stuart} (for the 3DVAR Gaussian filter), and \cite{Bessaih-Olson-Titi} (which is using the determining parameters nudging approach of this paper for data assimilation), the previously mentioned theoretical work assumed that the observational measurements are error free. Notably, the authors of \cite{Hayden_Olson_Titi} present an algorithm for data assimilation that uses discrete in space and time measurements.

In \cite{Azouani_Olson_Titi}, a new approach was introduced based on an idea from control theory \cite{Azouani_Titi}. The new approach was motivated by the fact that, unlike the case of finite Fourier modes, measured data represents the values of solution on a discrete set of nodal points are not possible to insert directly into the various terms of the evolution equations involving spatial derivatives of such solution.  In the new algorithm,  rather than inserting the observational measurements directly into the equations, the authors in \cite{Azouani_Olson_Titi} introduced a feedback control term that forces (nudges) the model towards the reference solution corresponding to the observations. This is motivated by the fact that instabilities in dissipative evolution equations occur at the large spatial scales, hence the need to control, stabilize, or nudge these coarse scales in downscaling algorithms.

The downscaling algorithm for the general setting can be formally described as follows: suppose that $u(t)$ represents a solution of some dissipative dynamical system governed by an evolution equation of the type
\begin{align}\label{dissipative}
\od{u}{t} = F(u), 
\end{align}
where the initial data $u(0)= u_0$ is missing. Let $I_h(u(t))$ represent an interpolant operator based on the observational measurements of this system at a coarse spatial resolution of size $h$, for $t\in [ 0,T ]$. The algorithm proposed in \cite{Azouani_Olson_Titi} is to construct a solution $v(t)$ taking into account the observations that satisfies the evolution equations 
\begin{subequations}\label{du}
\begin{align}
&\od{v}{t} = F(v) - \mu (I_h(v)- I_h(u)), \\
&v(0)= v_0, 
\end{align}
\end{subequations}
where $\mu>0$ is a relaxation (nudging) parameter and $v_0$ is taken to be an arbitrary initial data. Notice that if system \eqref{du} is globally well-posed and $I_h(v)$ converge to $I_h(u)$ in time, then we recover the reference $u(x,t)$ from the approximate solution $v(x,t)$. The goal is to find estimates on $\mu>0$ and $h>0$, in terms of physical parameters of the evolution equation \eqref{dissipative}, such that the approximate solution $v(t)$ is with increasing accuracy to the reference solution $u(t)$ as more continuous data in time is supplied. After some large enough time $T>0$, the solution $v(T)$ can then be used as an initial condition in system \eqref{dissipative} to make future predictions of the reference solution $u(t)$ for  $t > T$, or one can continue with \eqref{du} itself, for as long more measurements are provided.

This algorithm was designed to work for general {\it dissipative} dynamical systems of the form \eqref{dissipative}. Such systems are known to have global, in time, solutions and a finite-dimensional global attractor (for this very reason we are still unable to establish any results concerning this algorithm for the 3D Navier-Stokes equations; nevertheless, the algorithm can still be implemented and tested in practice).   Moreover, these systems are also known to have finite set of determining parameters. A projection (onto say a finite number of low Fourier modes, or other types of interpolant projections based on nodal values and volume elements) is said to be determining if, whenever the projection of two trajectories of \eqref{dissipative} on the global attractor approach each other, as $t \to \infty$, the full trajectories approach each other, see, for example,
\cite{C_O_T, F_M_R_T, Foias_Prodi, Foias_Temam, Foias_Temam_2, Foias_Titi, Holst_Titi, Jones_Titi, Jones_Titi_2} and references therein. The estimates provided on $\mu>0$ and $h>0$ use the estimates for the global existing solution in the global attractor of the system, in terms of the physical parameters, such as Reynolds number, in the context of the Navier-Stokes equations, for example.

In the context of the incompressible 2D NSE, the authors in \cite{Azouani_Olson_Titi} studied the conditions under which the approximate solution $v(t)$, obtained by this algorithm of data assimilation, converges to the reference solution $u(t)$ over time. An extension of this approach to the case when the observational data contains stochastic noise was analyzed in \cite{Bessaih-Olson-Titi}.
Under these assumptions, the data assimilation algorithm consists of a system of stochastically forced
Navier-Stokes equations. The main result established in \cite{Bessaih-Olson-Titi} gives resolution conditions which guarantee that the
limit supremum, as the time tends to infinity, of the expected value
of the $L^2$-norm of the difference between the approximating solution
and the actual reference solution; i.e. the error, is bounded by an estimate involving the variance
of the noise in the measurements and the spatial resolution of the collected data, $h$. 

In addition to understanding how the noise in the data affects the accuracy of the prediction, another major problem in data assimilation is that typically the dimension of the observation vector is less than the dimension of the model's state vector.  For example, in order to improve hurricane prediction and typhoon forecasts it is important to get the water vapor observations into the system model.   Water vapor is an important factor in the genesis of tropical cyclones \cite{Anthes2008}, but accurate and timely measurements still remain a challenge with our current observing technology (some recent technological advances in the design of atmospheric water vapor observing technology like the GPS Radio Occultation are currently underway \cite{White-Paper}). In the mean time, it is important to analyze the validity and success of a data assimilation algorithm when some state variable observations are not available as an input on the numerical forecast model.

Our main idea in this work stems from the work in \cite{FJT} which proposes a continuous data assimilation scheme to the two-dimensional incompressible B\'enard convection problem. The 2D B\'enard convection problem is given by
\begin{subequations}\label{Bous}
\begin{align}
&\pd{u}{t} - \nu\Delta u + (u\cdot\nabla)u + \nabla p = \theta \mathbf{e}_2, \label{Bous1}\\
&\pd{\theta}{t} - \kappa\Delta\theta + (u\cdot\nabla)\theta - u\cdot\mathbf{e}_2 =0, \label{Bous2}\\
&\nabla\cdot u= 0,\label{Bous_div}\\
&u(0, x,y) = u_0(x,y), \quad \theta(0, x,y)=\theta_0(x,y), 
\end{align}
\end{subequations}
with appropriate boundary conditions. Here $\nu>0$ is the fluid viscosity, $\kappa>0$ is the diffusion coefficient, $\mathbf{e}_2 =(0,1)$ is the second standard basis vector in $\nR^2$. The unknowns are the fluid velocity $u(t,x,y) = (u_1(t, x,y), u_2(t, x,y))$, the fluid pressure $p(t, x,y)$, and the scalar function $\theta(t, x,y)$, which may be the fluctuation of the density of the fluid or the temperature of the fluid. 
The authors in \cite{FJT} proposed an algorithm for the construction of $\uu(t)$ and $\eta(t)$, that approximates the velocity $u$ and the temperature fluctuations $\theta$, respectively,  from the observational measurements $I_h(u(t))$ of the two components of the velocity field  (but without needing the measurements $I_h(\theta(t))$ for the temperature fluctuations), for $t\in [0,T]$ whose evolution is given by 
\begin{subequations}\label{DA_Bous}
\begin{align}
&\pd{\uu}{t} -\nu \Delta \uu + (\uu\cdot\nabla)\uu +\nabla P = \eta\mathbf{e}_2- \mu(I_h(\uu)-I_h(u)), \\
&\pd{\eta}{t} -\kappa\Delta\eta + (\uu\cdot\nabla)\eta - \uu\cdot\mathbf{e}_2 = 0, \\
&\nabla \cdot \uu = 0, \\
&\uu(0,x,y) = \uu_0(x,y), \quad \eta(0,x,y) = \eta_0(x,y),
\end{align}
\end{subequations}
with the corresponding appropriate boundary conditions; where $\mu$ is a positive relaxation (nudging) parameter, which relaxes the coarse spatial scales of $\uu$ toward  the observed data, $P$ is the approximate pressure, and 
$\uu_0, \eta_0$ are taken to be arbitrary. Notice that this algorithm is different from the general algorithm presented in \cite{Azouani_Olson_Titi}. The algorithm in \eqref{DA_Bous} construct the approximate solutions for the velocity $u$ and temperature fluctuations $\theta$ using only the observational data, $I_h(u)$, of the velocity and without any measurements for the temperature (or density) fluctuations.  Ideally, one would like to design an algorithm based on temperature measurements only, but an algorithm is still out of reach for \eqref{Bous}.

In this work, we introduce an abridged dynamic continuous data assimilation for the 2D NSE inspired by the recent algorithms introduced in \cite{Azouani_Olson_Titi, FJT}. We establish convergence results for the improved algorithm where the observational data needed to be measured and inserted into the model equation is reduced or subsampled.  Here we handle the idealized case where the measured observational data is assumed to be free of noise and that the model parameters are exact. To be more precise, to review some relevant literature and to set some notation, we start by recalling the 2D NSE which can be written as 
\begin{subequations}\label{2D_NSE_com}
\begin{align}
\pd{u_1}{t} - \nu\Dd u_1 +  u_1\pp_x u_1 + u_2\pp_yu_1 + \pp_x p &= f_1, \\
\pd{u_2}{t} - \nu\Dd u_2 +  u_1\pp_x u_2 + u_2\pp_yu_2  + \pp_y p &= f_2, \\
\pp_xu_1 + \pp_yu_2 &=0,\label{div_u}\\
u_1(0,x,y) = u_1^0(x,y), \quad u_2(0,x,y) &= u_2^0(x,y),
\end{align}
\end{subequations}
where $(u_1(t, x,y),u_2(t,x,y))$ is the  velocity of the fluid at time $t$ and position $(x,y)\in \Omega$, $\nu>0$ represents the kinematic viscosity, $p(t,x,y)$ is the pressure and $(f_1(x,y),f_2(x,y))$ is the body force applied to the fluid, which we assume to be time independent. We will consider system \eqref{2D_NSE_com} in a physical domain $\Om$, with either no-slip boundary conditions or periodic boundary conditions. In the case of no-slip Dirichlet boundary conditions we take $u=0$ on $\pp\Om$. The domain $\Om$ is an open, bounded and connected domain of $\nR^2$ with $C^2$ boundary. In the case of periodic boundary conditions we require $u$, $p$ and $f$ to be $L$-periodic, in both $x$ and $y$ directions, with zero spatial averages over the fundamental periodic domain $\Om=[0,L]^2$.  

Inspired by \cite{FJT}, our proposed downscaling algorithm for the construction of approximate solution, $\uu(t,x,y)$  from the observational coarse measurements of the one component of the velocity, e.g., the second component, $I_h(u_2(t))$, for the reference solution $u(t,x,y)$, for $t\in [0,T]$ is given by 
\begin{subequations}\label{DA_NSE}
\begin{align}
\pd{\uu_1}{t} - \nu\Dd \uu_1 +  \uu_1\pp_x\uu_1 + \uu_2\pp_y\uu_1 + \pp_x \p &= f_1, \\
\pd{\uu_2}{t} - \nu\Dd \uu_2 +  \uu_1\pp_x\uu_2 + \uu_2\pp_y\uu_2  + \pp_y \p &= f_2 - \mu (I_h(U_2)-I_h(u_2)), \\
\pp_x\uu_1 + \pp_y\uu_2 &=0, \label{div_uu} \\
\uu_1(0,x,y) = \uu_1^0(x,y), \quad \uu_2(0,x,y) &= \uu_2^0(x,y). 
\end{align}
\end{subequations}
Here, $\mu$ is a positive nudging parameter, which relaxes (nudges) the coarse spatial scales of $\uu_2$ toward those of the observed data $I_h(u_2)$, $\p$ is the approximate pressure. A choice for $\uu_1^0$ and $\uu_2^0$ is arbitrary, and can be simply taken to be $\uu_1^0=0$ and $\uu_2^0=0$. If we knew $u_1^0$ and $u_2^0$, then we could take $\uu_1^0=u_1^0$ and  $\uu_2^0=u_2^0$ and the solution $(\uu_1,\uu_2)$ will be identically $(u_1, u_2)$, by the uniqueness of solutions of system \eqref{DA_NSE}.  But, $u_1^0$ and $u_2^0$ are not available, which is the main reason for introducing data assimilation algorithms.     We note that this algorithm requires observational measurements of only one component of the velocity vector field, horizontal $I_h(u_1(t))$ or vertical $I_h(u_2(t))$. Here, the observational measurements $I_h(u_2(t))$ were chosen as an example. 

We will consider interpolant observables given by linear interpolant operators $I_h: \Hph{1} \rightarrow \Lph{2}$, that approximate identity and satisfy the approximation property 
\begin{align}\label{app} 
\norm{\varphi - I_h(\varphi)}_{\Lph{2}} \leq \gamma_0h\norm{\varphi}_{\Hph{1}}, 
\end{align}
for some positive constant $\gamma_0$ and for every $\varphi$ in the Sobolev space $\Hph{1}$.   
We also consider a second type of interpolant  observables given by linear interpolant operators $I_h: \Hph{2}\rightarrow\Lph{2}$, that satisfy the approximation property 
\begin{align}\label{app2} 
\norm{\varphi - I_h(\varphi)}_{\Lph{2}} \leq \gamma_1h\norm{\varphi}_{\Hph{1}} + \gamma_2h^2\norm{\varphi}_{\Hph{2}}, 
\end{align}
for some positive constants $\gamma_1, \ \gamma_2$ and for every $\varphi$ in the Sobolev space $\Hph{2}$. One example of an interpolant observable that satisfies \eqref{app} is the orthogonal projection onto the low Fourier modes with wave numbers $k$ such that $|k|\leq 1/h$. A more physical example are the volume elements that were studied in \cite{Jones_Titi}. An example of an interpolant observable that satisfies \eqref{app2} is given by the measurements at a discrete set of nodal points in $\Omega$ (see Appendix A in \cite{Azouani_Olson_Titi}). We will call the interpolants that satisfy \eqref{app} and \eqref{app2} of type I and type II, respectively.

We provide explicit estimates on the spatial resolution $h$ of the observational measurements and the relaxation (nudging) parameter $\mu$, in terms of physical parameters, that are needed in order for the proposed downscaling algorithm to recover the reference resolution.
While the typical scenario in data assimilation is to choose $\mu$ depending on $h$, in our convergence analysis we choose our parameters $\mu$ and $h$ to depend on physical parameters. More explicitly, we choose $\mu$ to depend on the bounds of the solution on the global attractor of the system and then choose $h$ to depend on $\mu$ and the physical parameters. The philosophy here is that in order to prove the convergence theorems, we need to have a complete resolution of the flow, so $h$ has to depend on the physical parameters a.k.a the Grashof (Reynolds) number. Numerical simulations in \cite{Gesho} (see also \cite{Hayden_Olson_Titi}) have shown that, in the absence
of measurements errors, the continuous data assimilation algorithm \eqref{du} for the 2D Navier-Stokes equations performs much better than analytical estimates in \cite{Azouani_Olson_Titi} would suggest. This was also noted in a different context in \cite{Olson_Titi_2003} and \cite{Olson_Titi_2008}. It is likely that the data assimilation algorithm studied in this paper will also perform much better than suggested by the analytical results, i.e. under more relaxed conditions than those assumed in the rigorous estimates. This is a subject of future work. 

We can extend the corresponding convergence analysis for the 2D B\'enard equation,  where the approximate solutions  constructed using observations in {\it only one component of the two-dimensional velocity field and without any measurements on the temperature}, converge in time to the reference solution of the 2D B\'enard system. This will be a progression of a recent result in \cite{FJT} where convergence results were established, given that observations are known at discrete points on {\it all} of the components of the velocity field and without any measurements of the temperature. The proposed data assimilation algorithm can also be applied to several three-dimensional subgrid scale turbulence models. In \cite{ALT2014}, it was shown that approximate solutions constructed using observations on all three components of the unfiltered velocity field converge in time to the reference solution of the 3D NS-$\alpha$ model.   Morever, in \cite{MTT} a similar data assimilation algorithm was introduced for the 3D Brinkman-Forchheimer-Darcy model for flow in porous media.   We give a progression to this scheme and propose that one can show a sharpened results that the approximate solutions  constructed using observations in {\it only any two components and without any measurements on the third component of the velocity field} converge in time to the reference solution for this model. The analysis for both applications are work in progress. Notably, a similar algorithm for stochastically noisy data is at hand combining ideas from the present work and \cite{Bessaih-Olson-Titi}. 

In the section \ref{pre}, we lay out the functional setting commonly used in the mathematical study of the Navier-Stokes equations. The main results are in sections \ref{conv_typeI} and \ref{conv_typeII}. We find estimates on the adequate resolution in the observational data $h$ and the relaxation parameter $\mu$ for observational measurements that satisfy \eqref{app} and \eqref{app2}, separately, in the cases of Dirichlet and periodic boundary conditions. We prove the well-posedness of system \eqref{DA_NSE}
as well as the convergence (at an exponential rate) of the approximate solution $v(t)$ of \eqref{DA_NSE}  to the reference solution $u(t)$ to the 2D Navier-Stokes equations \eqref{2D_NSE_com}. 
\bigskip

%---------------------------------------------------------------------
\section{Preliminaries}\label{pre}

For the sake of completeness, this section presents some preliminary material and notation commonly used in the mathematical study of hydrodynamics models, in particular in the study of the Navier-Stokes equations (NSE) and the Euler equations. For more detailed discussion on these topics, we refer the reader to, e.g., \cite{Constantin_Foias_1988}, \cite{Robinson}, \cite{Temam_1995_Fun_Anal} and \cite{Temam_2001_Th_Num}. 

In the two-dimensional case with no-slip Dirichlet boundary conditions, let $\Omega$ be an open, bounded and connected domain with $C^2$ boundary. We define $\mathcal{V}$ to be the set of divergence free and compactly supported $C^\infty$ vector fields from $\Om\subset\nR^2 \rightarrow \nR^2$. In the case of periodic boundary conditions, let $\Omega= [0,L]^2$ for some fixed $L>0$,  we define $\mathcal{V}$ to be the set of all $L$-periodic trigonometric polynomials from $\nR^2$ to $\nR^2$ that are divergence free and have zero averages. We denote by $\Lph{p}$, $\Wph{s}{p}$, and $\Hph{s}\equiv \Wph{s}{2}$ to be the usual Sobolev spaces in two-dimensions. We will denote by $H$ and $V$ the closure of $\mathcal{V}$ in $\Lph{2}$ and $\Hph{1}$, respectively. 
We also denote by $\Hphd{1}$ the set of $\Hph{1}$ functions with zero traces at the boundary $\pp \Omega$.  In the periodic case we denote by $\Hphp{1}$ the subspace of $\Hph{1}$ of functions that are periodic with zero average. 

We define the inner products on $\Lph{2}$ and $H$ by
\[(u,w)=\sum_{i=1}^2\int_{\Om} u_iw_i\,dxdy,
\]
and the associated norm $\norm{u}_{\Lph{2}}=(u,u)^{1/2}$. Notice that $V$, $\Hphd{1}$ and $\Hphp{1}$ are Hilbert spaces with the inner product 
\[((u,w))=\sum_{i,j=1}^2\int_{\Om}\partial_ju_i\partial_jw_i\,dxdy, 
\]
with the associated $\norm{u}_{\Hph{1}}=((u,u))^{1/2}= \norm{\nabla u}_{\Lph{2}}$. Note that $((\cdot,\cdot))$ defines a norm due to the Poincar\'e inequality \eqref{poincare}, below. 

\begin{remark}
We will use these notations indiscriminately for both scalars and vectors, which should not be a source of confusion. 
\end{remark}

We denote the dual of $V$ by $V^{'}$. Define the Leray projector $P_\sigma$ as the orthogonal projection from $\Lph{2}$ onto $H$, and define the Stokes operator $A:V \rightarrow V^{'}$ is given by 
$$ Au = -P_\sigma \Delta u, $$
with domain  $\mathcal{D}(A)= V \cap \Hph{2}$. The linear operator $A$ is self-adjoint and positive definite with compact inverse $A^{-1}: H \rightarrow H$. Thus, there exists a complete orthonormal set of eigenfunctions $w_i$ in $H$ such that $Aw_i= \lambda_iw_i$ where $0<\lambda_1\leq\lambda_2\leq . . . \leq\lambda_i\leq\lambda_{i+1}\leq . . . $ for $i\in \nN$. 

We also recall the Poincar\'e inequalities:
\begin{subequations}\label{poincare}
\begin{enumerate}
\item For all $\vphi \in V$:
\begin{align}\label{poincare_1}
\|\vphi\|_{\Lph{2}}^2\leq \lambda_1^{-1}\|\nabla\vphi\|_{\Lph{2}}^2,  
\end{align}
\item for all $\vphi \in \mathcal{D}(A)$: 
\begin{align}\label{poincare_2}
\|\nabla\vphi\|_{\Lph{2}}^2\leq \lambda_1^{-1}\|\Delta\vphi\|_{\Lph{2}}^2, 
\end{align}
\end{enumerate}
\end{subequations}
where $\lambda_1$ is the smallest eigenvalue of the operator $A$ in two-dimensions, subject to the relevant boundary conditions. 

Let $Y$ be a Banach space.  We denote by $L^p([0,T];Y)$ the space of (Bochner) measurable functions $t\mapsto w(t)$, where $w(t)\in Y$, for a.e. $t\in[0,T]$, such that the integral $\int_0^T\|w(t)\|_Y^p\,dt$ is finite. 

Hereafter, $c$, $c_L$ and $c_T$ will denote universal dimensionless positive constants. Our estimates for the nonlinear terms will involve the Ladyzhenskaya's inequality in two-dimensions for an integrable function $\vphi \in V$: 
\begin{equation}\label{L4_to_H1}
\|\vphi\|_{\Lph{4}}^2\leq c_L\|\vphi\|_{\Lph{2}}\|\nabla\vphi\|_{\Lph{2}}. 
\end{equation}

\begin{remark}
We note that the Poincar\'e inequality \eqref{poincare_1} and the Ladyzenskaya inequality \eqref{L4_to_H1} hold for $\varphi \in \Hphd{1}$ or $\varphi \in \Hphp{1}$. Similarly the Poincar\'e inequality \eqref{poincare_2} holds for $\varphi \in \Hph{2}\cap \Hphd{1}$ or $\varphi \in \Hph{2}\cap \Hphp{1}$. We will use these versions of the Poincar\'e inequality and the Ladyzenskaya inequality for the components of the velocity vector field. 
\end{remark} 

Also, we will use the following logarithmic estimates for the nonlinear term in two-dimensions. These estimates are the analogue of the logarithmic estimates proved in \cite{Titi_1986}, for the advection nonlinear term of the Navier-Stokes equations; and they can be proved following the same steps of the proof in \cite{Titi_1986}. 
\begin{subequations}
\begin{enumerate}
\item 
For every $u, v, w \in \Hphd{1}$ (or $\Hphp{1}$), with $w\neq 0$, we have  
\begin{align}\label{titi_1}
& \abs{\int_\Omega u\,\pp_iv\, w\, dxdy} \leq \notag \\
& \quad c_T\norm{\nabla u}_{\Lph{2}}\norm{\nabla v}_{\Lph{2}}\norm{w}_{\Lph{2}}\left(1 + \log \left(\frac{\norm{\nabla w}_{\Lph{2}}}{\lambda_1^{1/2}\norm{w}_{\Lph{2}}}\right)\right)^{1/2}. 
\end{align}
\item For every $u \in \Hphd{1}$ and $v, w \in \Hph{2}\cap\Hphd{1}$ (or $u \in \Hphp{1}$ and $v, w \in \Hph{2}\cap\Hphp{1}$), with $v\neq 0$, we have  
\begin{align}\label{titi_2}
&  \abs{\int_\Omega u\, \pp_iv\, \pp_{jj}w\, dxdy} \leq \notag \\
&\quad c_T\norm{\nabla u}_{\Lph{2}}\norm{\nabla v}_{\Lph{2}}\norm{\Delta w}_{\Lph{2}}\left(1 + \log \left(\frac{\norm{\Delta v}_{\Lph{2}}}{\lambda_1^{1/2}\norm{\nabla v}_{\Lph{2}}}\right)\right)^{1/2}. 
\end{align}
\item For every $v\in \Hphd{1}$ and $u, w \in \Hph{2}\cap \Hphd{1}$ (or $u \in \Hphp{1}$ and $v, w \in \Hph{2}\cap\Hphp{1}$), with $u\neq 0$, we have  
\begin{align}\label{titi_3}
& \abs{\int_\Omega u\, \pp_iv\, \pp_{jj}w\, dxdy}\leq \notag \\
&\quad c_T\norm{\nabla u}_{\Lph{2}}\norm{\nabla v}_{\Lph{2}}\norm{\Delta w}_{\Lph{2}}\left(1 + \log \left(\frac{\norm{\Delta u}_{\Lph{2}}}{\lambda_1^{1/2}\norm{\nabla u}_{\Lph{2}}}\right)\right)^{1/2},  
\end{align}
\end{enumerate}
\end{subequations} 
where $\partial_i$ is interchangeable with $\pp_x$ or $\pp_y$, and $\pp_{jj}$ is interchangeable with $\pp_{xx}$ or $\pp_{yy}$.  We note that the logarithmic estimate \eqref{titi_3} can follow by an argument using the Brezis-Gallouet logarithmic inequality \cite{Brezis_Gallouet_1980}. In \cite{Titi_1986}, it was proven using a different approach. On the other hand, the logarithmic estimate \eqref{titi_2} does not follow as a consequence of the Brezis-Gallouet inequality (see \cite{Titi_1986} for the proof). 

We recall that in two dimensions and in the case of periodic boundary conditions the nonlinearity also satisfies 
\begin{subequations}
\begin{align}\label{per_orth}
((u\cdot\nabla)u, \Delta u) = 0, 
\end{align}
for each $u\in \mathcal{D}(A)$ and consequently 
\begin{align}\label{per_orth_2}
((u\cdot\nabla)w,\Delta w) + ((w\cdot\nabla)u, \Delta w) + ((w\cdot\nabla)w, \Delta u) = 0, 
\end{align}
\end{subequations}
for each $u$ and $w\in \mathcal{D}(A)$.

Furthermore, inequality \eqref{app} implies that
\begin{align}\label{app_F}
 \norm{w-I_h(w)}_{\Lph{2}}^2\leq c_0^2h^2\norm{\nabla w}_{\Lph{2}}^2,
\end{align}
for every $w\in V,\;  \Hphd{1}$, or $\Hphp{1}$, where $c_0=\gamma_0$, and respectively, \eqref{app2} implies that
\begin{align}\label{app2_F}
\norm{w-I_h(w)}_{\Lph{2}}^2\leq \frac{1}{2}c_0^2h^2\norm{\nabla w}_{\Lph{2}}^2+ \frac14c_0^4h^4\norm{\Delta w}_{\Lph{2}}^2, 
\end{align}
for every $w\in \mathcal{D}(A), \; \Hphd{1}\cap\Hph{2}$, or $\Hphp{1}\cap\Hph{2}$, for some $c_0>0$ that depends only on $\gamma_0$, $\gamma_1$ and $\gamma_2$. We note that in the case of periodic boundary conditions, we demand that the spatial average of $I_h(w)$ to be zero, for every $w$ in the relevant domain of $I_h$ (c.f. \cite{Azouani_Olson_Titi}). This is to guarantee that the spatial average of the solution $\uu$ of \eqref{DA_NSE} is preserved, and hence can be chosen to be zero. 

We will use the following elementary inequality proved in \cite{Azouani_Olson_Titi}.
\begin{lemma}\label{log_prop}
Let $\phi(r) = r - \gamma(1+\log(r))$, where $\gamma>0$. Then
\begin{align*}
\min\{\phi(r): r\geq 1\} \geq -\gamma \log(\gamma). 
\end{align*}
\end{lemma}
We will also apply the following uniform Gronwall's inequality proved in \cite{Jones_Titi}. 
\begin{lemma}\label{gen_gron_2}[Uniform Gronwall's inequality] Let $\tau>0$ be arbitrary but fixed. Suppose that $Y(t)$ is an absolutely continuous function which is locally integrable and that it satisfies the following:
\begin{align*}
\od{Y}{t} + \tilde{\beta}(t)Y \leq 0,\qquad \text{ a.e. on } (0,\infty), 
\end{align*}
and 
\begin{align}\label{uniform_conditions}
\liminf_{t\rightarrow\infty} \int_t^{t+\tau} \tilde{\beta}(s)\,ds \geq \gamma, \qquad 
\limsup_{t\rightarrow\infty} \int_t^{t+\tau} \tilde{\beta}^{-}(s)\,ds < \infty, 
\end{align}
for some $\gamma>0$, where $\tilde{\beta}^{-} = \max\{-\tilde{\beta}, 0\}$.  
Then, $Y(t)\rightarrow 0$ at an exponential rate, as $t\rightarrow \infty$. 
\end{lemma}

Here we denote by $G$ the Grashof number in two-dimensions
\begin{align}\label{Grashof_2}
G = \frac{1}{\nu^2\lambda_1} \norm{f}_{\Lph{2}}. 
\end{align}

We recall that the 2D NSE \eqref{2D_NSE_com} are well-posed and posses a finite-dimensional global attractor when $f$ is time-independent, see e.g., \cite{Constantin_Foias_1988}, \cite{Robinson}, \cite{Temam_1997}. Next, we give bounds on solutions $u$ of \eqref{2D_NSE_com} that we will use later in our analysis. These bounds are proved in the references listed above.
The estimate \eqref{jolly} is proved in \cite{DFJ4} and estimate \eqref{dirichlet_H2} will be proved in Appendix A. 

\begin{proposition}\label{unif_bounds_2D_NSE}
Let $\tau>0$ be arbitrary, and let $G$ be the Grashof number given in \eqref{Grashof_2}. Suppose that $u$ is a solution of \eqref{2D_NSE_com} subject to no-slip Drichlet boundary conditions, then there exists a time $t_0>0$ such that for all $t\geq t_0$ we have
\begin{subequations}
\begin{align}\label{dirichlet_L2}
\norm{u(t)}_{H}^2 \leq 2\nu^2G^2, \quad \quad \int_t^{t+\tau} \norm{\nabla u(s)}_{\Lph{2}}^2\,ds \leq 2(1+\tau\nu\lambda_1)\nu G^2,
\end{align} 
\begin{align}\label{dirichlet_H1}
\norm{\nabla u(t)}_{\Lph{2}}^2 \leq {\tilde c}\nu^2\lambda_1G^2e^{G^4},  
\end{align}
\begin{align}\label{dirichlet_int_H2}
\int_{t}^{t+\tau}\norm{\Delta u(s)}_{\Lph{2}}^2\, ds \leq ({\tilde c}e^{G^4} + \tau \nu\lambda_1)\nu\lambda_1 G^2,  
\end{align}
and 
\begin{align}\label{dirichlet_H2}
\norm{\Delta u(t)}_{\Lph{2}}^2 \leq {\tilde c} \nu^2\lambda_1^2G^2\left(1+(1+ G^2e^{G^4})(1+e^{G^4}+G^4e^{G^4}) \right), 
\end{align}
\end{subequations}
for some positive non-dimensional constant ${\tilde c}$. 
In the case of periodic boundary conditions we have 
\begin{subequations}
\begin{align}\label{per_est}
\norm{\nabla u(t)}_{\Lph{2}}^2 \leq 2\nu^2\lambda_1G^2, \quad \int_t^{t+\tau} \norm{\Delta u(s)}_{\Lph{2}}^2\,ds \leq 2(1+\tau\nu\lambda_1)\nu \lambda_1G^2,
\end{align}
and 
\begin{align}\label{jolly}
\norm{\Delta u(t)}_{\Lph{2}}^2 \leq {\tilde c}\nu^2\lambda_1^2(1+G)^4. 
\end{align}
\end{subequations}
\end{proposition}

\begin{remark}\label{t0_remark}
In this work, we will assume that the reference solution of the Navier-Stokes equations, that we are trying to approximate, has evolved enough in time to satisfy the estimates provided in the above proposition. That is, we will assume that the solution satisfies these estimates at $t=0$. For that reason, we will take, without loss of generality, $t_0=0$.
\end{remark}
\bigskip
%---------------------------------------------------------------------
\section{Convergence analysis with observable data of type I}\label{conv_typeI}
%----------------------------------------------------------------------

Following similar techniques introduced for the two-dimensional Navier-Stokes equations (see, e.g., \cite{Constantin_Foias_1988, Robinson, Temam_2001_Th_Num}), we can prove the global well-posedness of system \eqref{DA_NSE} as stated below, when the observable data satisfy \eqref{app}. For more details, see \cite{Azouani_Olson_Titi}.

\begin{theorem}\label{exist_uniq_NSE_1}[Well-posedness of solutions] Suppose $I_h$ satisfy \eqref{app} with $\mu>0$ and $h>0$ are chosen such that $\mu c_0^2h^2\leq \nu$, where $c_0$ is the constant in \eqref{app}. If the initial data $\uu_0 \in V$, then the continuous data assimilation system \eqref{DA_NSE}, subject to Dirichlet or periodic boundary conditions, possess a unique global strong solution $\uu(t, x,y)= (\uu_1(t, x,y), \uu_2(t , x,y))$ that satisfies
\begin{align*}
\uu \in C([0,T]; V) \cap L^2([0,T];{\mathcal{D}(A)}), \quad \text{and} \quad \od{\uu}{t} \in L^2([0,T];H). 
\end{align*}
Moreover, the solution $\uu(t, x,y)$ depends continuously on the initial data $\uu_0$. 
\end{theorem} 
We will now state and  prove that under certain conditions on $\mu$ and $h$, the solution $(\uu_1,\uu_2)$ of the data assimilation system \eqref{DA_NSE} converges to the solution $(u_1,u_2)$ of the two-dimensional Navier-Stokes equations \eqref{2D_NSE_com}, subject to periodic or Dirichlet boundary conditions, respectively, as $t\rightarrow \infty$, when the observable operators satisfy \eqref{app}. 

\begin{theorem}\label{th_conv_NS_1}
Suppose that $I_h$ satisfy the approximation property \eqref{app} and $u(t,x,y)$ $=$ $(u_1(t,x,y),u_2(t,x,y))$ is a strong solution in the global attractor of \eqref{2D_NSE_com} subject to Dirichlet boundary conditions. Let $\uu(t,x,y)$ $=$ $(\uu_1(t,x,y),\uu_2(t,x,y))$ be a strong solution of \eqref{DA_NSE}, subject to Dirichlet boundary conditions. If $\mu>0$ is chosen large enough such that
\begin{align}\label{mu_2D_NSE_1}
\mu \geq 2 c\nu\lambda_1( 1+ \log(G) + G^4) G^2, 
\end{align}
and $h>0$ is chosen small enough such that $\mu c_0^2h^2\leq \nu $,  then $\norm{u(t)-\uu(t)}_{\Lph{2}}^2~\rightarrow~0$ at an 
exponential rate, as $t \rightarrow \infty$.
\end{theorem}

\begin{proof}
Define $\dU= u-\uu$ and $\dP = p-\p$. Then $\dU_1$ and $\dU_2$ satisfy the equations
\begin{subequations}\label{dU_NSE}
\begin{align}
&\pd{\dU_1}{t} - \nu \Dd \dU_1+ \uu_1\pp_x\dU_1 + \uu_2\pp_y\dU_1 + \dU_1 \pp_xu_1 +\dU_2\pp_yu_1 + \pp_x\dP= 0, \label{dU_1}\\
&\pd{\dU_2}{t} - \nu \Dd \dU_2+ \uu_1\pp_x\dU_2 + \uu_2\pp_y\dU_2 + \dU_1 \pp_xu_2 +\dU_2\pp_yu_2 + \pp_y\dP= -\mu I_h(\dU_2), \label{dU_2}\\
&\pp_x\dU_1 + \pp_y\dU_2  =0. \label{div_dU}
\end{align}
\end{subequations}
Since $\dU_1$, $\dU_2$, $\od{\dU_1}{t}$ and $\od{\dU_2}{t}$ are bounded in $L^2([0,T];H)$, we can take the $\Lph{2}$ inner product of \eqref{dU_1} and \eqref{dU_2} with $\dU_1$ and $\dU_2$, respectively. We obtain, using the divergence free condition \eqref{div_uu}, integration by parts and using the relevant boundary conditions, that
\begin{align*}
\frac 12 \od{}{t} \norm{\dU_1}_{\Lph{2}}^2 + \nu \norm{\nabla\dU_1}_{\Lph{2}}^2 &\leq \abs{J_{1a}} + \abs{J_{1b}} - (\pp_x\dP,\dU_1), \\ 
\frac 12 \od{}{t} \norm{\dU_2}_{\Lph{2}}^2 + \nu \norm{\nabla \dU_2}_{\Lph{2}}^2 &\leq \abs{J_{2a}} + \abs{J_{2b}} - (\pp_y\dP,\dU_2) - \mu (I_h(\dU_2),\dU_2),  
\end{align*}
where
\begin{align*}
J_{1a} := (\dU_1 \pp_xu_1 , \dU_1), \qquad J_{1b}: = (\dU_2\pp_yu_1,\dU_1), \\
J_{2a} : = (\dU_1 \pp_xu_2, \dU_2), \qquad J_{2b} : = (\dU_2\pp_yu_2,\dU_2). 
\end{align*}
To estimate the nonlinear terms we proceed as follows: using integration by parts twice, we have   
\begin{align*}
J_{1a} = \left( \pp_xu_1,(\dU_1)^2\right) & = -2\left(u_1\dU_,\pp_x\dU_1\right)\notag\\ 
& = 2 \left(u_1\dU_1,\pp_y\dU_2\right) \notag \\ 
& = -2 \left(\dU_1\pp_yu_1, \dU_2\right) - 2\left(u_1 \pp_y\dU_1,\dU_2\right) \notag \\
&=: -2 (J_{1a1}) - 2 (J_{1a2}),
\end{align*}
where we used above the divergence free condition \eqref{div_dU} and the relevant boundary conditions. 
By the logarithmic estimate \eqref{titi_1} and Young's inequality, we have 
\begin{align}\label{a}
|J_{1a1}| &: = |\left(\dU_1\pp_yu_1, \dU_2\right)|\notag\\
&\leq c_T\norm{\nabla\dU_1}_{\Lph{2}}\norm{\nabla u_1}_{\Lph{2}}\norm{\dU_2}_{\Lph{2}}\left(1 + \log \left(\frac{\norm{\nabla \dU_2}_{\Lph{2}}}{\lambda_1^{1/2}\norm{\dU_2}_{\Lph{2}}}\right)\right)^{1/2}\notag\\
& \leq \frac{\nu}{64} \norm{\nabla \dU_1}^2_{\Lph{2}} + \frac{ c}{\nu}\norm{\nabla u_1}_{\Lph{2}}^2\left(1 + \log \left(\frac{\norm{\nabla \dU_2}_{\Lph{2}}}{\lambda_1^{1/2}\norm{\dU_2}_{\Lph{2}}}\right)\right)\norm{\dU_2}_{\Lph{2}}^2,
\end{align}
and 
\begin{align}\label{b}
|J_{1a2}|& : = |\left(u_1 \pp_y\dU_1,\dU_2\right)| \notag \\
& \leq \frac{\nu}{64} \norm{\nabla \dU_1}^2_{\Lph{2}} + \frac{ c}{\nu}\norm{\nabla u_1}_{\Lph{2}}^2\left(1 + \log \left(\frac{\norm{\nabla \dU_2}_{\Lph{2}}}{\lambda_1^{1/2}\norm{\dU_2}_{\Lph{2}}}\right)\right)\norm{\dU_2}_{\Lph{2}}^2.
\end{align}
Thus, 
\begin{align}\label{1}
\abs{J_{1a}} &\leq \frac{\nu}{32} \norm{\nabla \dU_1}_{\Lph{2}}^2 + \frac{ c}{\nu}\norm{\nabla u_1}_{\Lph{2}}^2\left(1 + \log \left(\frac{\norm{\nabla \dU_2}_{\Lph{2}}}{\lambda_1^{1/2}\norm{\dU_2}_{\Lph{2}}}\right)\right)\norm{\dU_2}_{\Lph{2}}^2.
\end{align}

A similar argument as in \eqref{a} and \eqref{b} yields 
\begin{align}\label{2}
\abs{J_{1b}} & : = |(\dU_2\pp_yu_1,\dU_1)| \notag \\
&\leq \frac{\nu}{32} \norm{\nabla \dU_1}_{\Lph{2}}^2 + \frac{ c}{\nu}\norm{\nabla u_1}_{\Lph{2}}^2\left(1 + \log \left(\frac{\norm{\nabla \dU_2}_{\Lph{2}}}{\lambda_1^{1/2}\norm{\dU_2}_{\Lph{2}}}\right)\right)\norm{\dU_2}_{\Lph{2}}^2. 
\end{align}
and 
\begin{align}\label{3}
\abs{J_{2a}} &: = |(\dU_1 \pp_xu_2, \dU_2)|\notag \\
&\leq \frac{\nu}{32} \norm{\nabla \dU_1}_{\Lph{2}}^2 + \frac{ c}{\nu}\norm{\nabla u_2}_{\Lph{2}}^2\left(1 + \log \left(\frac{\norm{\nabla \dU_2}_{\Lph{2}}}{\lambda_1^{1/2}\norm{\dU_2}_{\Lph{2}}}\right)\right)\norm{\dU_2}_{\Lph{2}}^2.  
\end{align}
By H\"older's inequality,  Ladyzhenskaya's inequality \eqref{L4_to_H1}, we have 
\begin{align}\label{4}
\abs{J_{2b}}= \abs{(\dU_2\pp_yu_2,\dU_2)} & \leq \norm{\dU_2}_{\Lph{4}}^2\norm{\pp_yu_2}_{\Lph{2}}\notag \\
& \leq c_L\norm{\dU_2}_{\Lph{2}}\norm{\nabla\dU_2}_{\Lph{2}}\norm{\pp_yu_2}_{\Lph{2}}\notag \\
& \leq \frac{\nu}{32} \norm{\nabla\dU_2}_{\Lph{2}}^2+ \frac{c}{\nu}\norm{\pp_yu_2}_{\Lph{2}}^2\norm{\dU_2}_{\Lph{2}}^2. 
\end{align}
Thanks to the assumption $\mu c_0^2h^2\leq \nu$ and Young's inequality, 
\begin{align}\label{5}
-\mu(I_h(\dU_2),\dU_2) & = -\mu(I_h(\dU_2)-\dU_2, \dU_2) - \mu \norm{\dU_2}_{\Lph{2}}^2\notag \\
& \leq \mu \norm{I_h(\dU_2)-\dU_2}_{\Lph{2}}\norm{\dU_2}_{\Lph{2}} - \mu \norm{\dU_2}_{\Lph{2}}^2\notag \\
&\leq \mu c_0h\norm{\dU_2}_{\Lph{2}}\norm{\nabla\dU_2}_{\Lph{2}} - \mu \norm{\dU_2}_{\Lph{2}}^2\notag \\
& \leq \frac{\mu c_0^2h^2}{2}\norm{\nabla\dU_2}_{\Lph{2}}^2 - \frac{\mu}{2}\norm{\dU_2}_{\Lph{2}}^2\notag \\
&\leq \frac{\nu}{2}\norm{\nabla\dU_2}_{\Lph{2}}^2 -\frac{\mu}{2}\norm{\dU_2}_{\Lph{2}}^2. 
\end{align}
Also, note that
\begin{align}\label{6}
(\pp_x\dP, \dU_1) + (\pp_y\dP, \dU_2) = 0, 
\end{align}
thanks to the divergence free condition \eqref{div_dU}, integration by parts and using the relevant boundary conditions.  
It follows from the estimates \eqref{1}--\eqref{6} that 
\begin{align*}
& \od{}{t} \norm{\dU}_{\Lph{2}}^2 + \frac\nu2\norm{\nabla \dU}_{\Lph{2}}^2 \leq \notag \\
& \qquad \qquad \left(\frac{ c}{\nu}\norm{\nabla u}_{\Lph{2}}^2\left(1 + \log \left(\frac{\norm{\nabla \dU_2}_{\Lph{2}}}{\lambda_1^{1/2} \norm{\dU_2}_{\Lph{2}}}\right)\right)-\mu\right)\norm{\dU_2}_{\Lph{2}}^2,  
\end{align*}
or 
\begin{align}\label{conv_est_1}
& \od{}{t} \norm{\dU}_{\Lph{2}}^2 + \frac\nu4\norm{\nabla \dU}_{\Lph{2}}^2 + \frac\nu4\norm{\nabla \dU_2}_{\Lph{2}}^2 \leq \notag \\
& \qquad \qquad \left(\frac{ c}{\nu}\norm{\nabla u}_{\Lph{2}}^2\left(1 + \log \left(\frac{\norm{\nabla \dU_2}_{\Lph{2}}^2}{\lambda_1 \norm{\dU_2}_{\Lph{2}}^2}\right)\right)-\mu\right)\norm{\dU_2}_{\Lph{2}}^2. 
\end{align}

Next, we will use a similar argument as in \cite{Azouani_Olson_Titi} and \cite{FJKT_2014} to prove that $\norm{\dU(t)}_{\Lph{2}}~\rightarrow~0$, exponentially, as $t \rightarrow \infty$. 

Using the Poincar\'e inequality \eqref{poincare}, we may rewrite \eqref{conv_est_1} as
\begin{align}\label{conv_est_11}
& \od{}{t} \norm{\dU}_{\Lph{2}}^2 + \frac{\nu\lambda_1}{4}\norm{\dU}_{\Lph{2}}^2 +  \frac{\nu\lambda_1}{4}\frac{\norm{\nabla \dU_2}_{\Lph{2}}^2}{\lambda_1\norm{\dU_2}_{\Lph{2}}^2}\norm{\dU_2}_{\Lph{2}}^2\leq \notag \\
& \qquad \qquad \left(\frac{ c}{\nu}\norm{\nabla u}_{\Lph{2}}^2\left(1 + \log \left(\frac{\norm{\nabla \dU_2}_{\Lph{2}}^2}{\lambda_1\norm{\dU_2}_{\Lph{2}}^2}\right)\right)-\mu\right)\norm{\dU_2}_{\Lph{2}}^2,  
\end{align}
or 
\begin{align}\label{conv_est_2}
& \od{}{t} \norm{\dU}_{\Lph{2}}^2 + \frac{\nu\lambda_1}{4}\norm{\dU}_{\Lph{2}}^2 +  \frac{\nu\lambda_1}{4}\left(\phi(r(t)) +\frac{4\mu}{\nu\lambda_1}\right)\norm{\dU_2}_{\Lph{2}}^2 \leq 0, 
\end{align}
where we denoted by 
$$\phi(r(t)) := r(t) - \gamma(t)(1+\log(r(t)),$$
$$ r(t): = \frac{\norm{\nabla \dU_2}_{\Lph{2}}^2}{\lambda_1 \norm{\dU_2}_{\Lph{2}}^2}, \quad \gamma(t) : = \frac{c}{\nu^2\lambda_1}\norm{\nabla u}_{\Lph{2}}^2. $$
Now, Lemma \ref{log_prop} implies that
$$\phi(r(t)) \geq - \gamma \log(\gamma) = - \frac{c}{\nu^2\lambda_1}\norm{\nabla u}_{\Lph{2}}^2\log\left(\frac{c}{\nu^2\lambda_1}\norm{\nabla u}_{\Lph{2}}^2\right).$$
Setting 
$$\beta(t) := \mu -   \frac{c}{\nu}\norm{\nabla u}_{\Lph{2}}^2\log\left(\frac{c}{\nu^2\lambda_1}\norm{\nabla u}_{\Lph{2}}^2\right), $$
we have 
\begin{align}
& \od{}{t} \norm{\dU}_{\Lph{2}}^2 + \frac{\nu\lambda_1}{4}\left(\norm{\dU_1}_{\Lph{2}}^2+\norm{\dU_2}_{\Lph{2}}^2\right)  + \beta(t) \norm{\dU_2}_{\Lph{2}}^2 \leq 0. 
\end{align}

We may conclude that 
\begin{align*}
\od{}{t} \norm{\dU}_{\Lph{2}}^2 +\min\{\frac{\nu\lambda_1}{4}, \,\beta(t)\}\norm{\dU}^2_{\Lph{2}}\leq 0. 
\end{align*}
Taking $\tau = (\nu\lambda_1)^{-1}$ in Proposition \ref{unif_bounds_2D_NSE}, using \eqref{dirichlet_L2} and \eqref{dirichlet_int_H2}, and due to Remark \ref{t0_remark}, we conclude that
\begin{align}\label{major_est_1}
 \frac{c}{\nu}\int_t^{t+\tau}&\norm{\nabla u(s)}_{\Lph{2}}^2\log\left(\frac{c}{\nu^2\lambda_1}\norm{\nabla u(s)}_{\Lph{2}}^2\right)\,ds \notag\\
  &\leq  \frac{c}{\nu}(1 + \log(G) + G^4) \int_t^{t+\tau} \norm{\nabla u(s)}_{\Lph{2}}^2\,ds \notag \\
 & \leq c( 1+ \log(G) + G^4) G^2, 
\end{align} 
for all $t\geq 0$. Therefore, the assumption \eqref{mu_2D_NSE_1} implies that 
\begin{align}
\liminf_{t\rightarrow\infty} \int_t^{t+\tau} \beta(s)\,ds \geq \frac{\mu}{2\nu\lambda_1}>0, 
\end{align}
and 
\begin{align}
\limsup_{t\rightarrow\infty} \int_t^{t+\tau} \beta(s)\,ds \leq\frac{3\mu}{2\nu\lambda_1} < \infty. 
\end{align}
Define $\tilde{\beta}(t) := \min\{\frac{\nu\lambda_1}{4}, \,\beta(t)\}$, then $\tilde{\beta}(t)$ satisfies the condition \eqref{uniform_conditions}. By the uniform Gronwall's  lemma \eqref{gen_gron_2}, we obtain 
$$\norm{\dU}_{\Lph{2}}^2= \norm{u(t)-\uu(t)}_{\Lph{2}}^2 \rightarrow 0,$$ at an exponential rate as $t\rightarrow\infty$.
\end{proof}

\begin{remark}\label{periodic_1}
Theorem \eqref{th_conv_NS_1} holds in the case of periodic boundary conditions. Using \eqref{per_est} in Proposition \ref{unif_bounds_2D_NSE}, the estimate \eqref{major_est_1} can be improved to 
\begin{align}
 \frac{c}{\nu}\int_t^{t+\tau}\norm{\nabla u(s)}_{\Lph{2}}^2\log\left(\frac{c}{\nu^2\lambda_1}\norm{\nabla u(s)}_{\Lph{2}}^2\right)\,ds 
 & \leq c( 1+ \log(G)) G^2, 
\end{align}
thus the lower bound on $\mu$ \eqref{mu_2D_NSE_1} can be improved to 
\begin{align}
\mu \geq 2 c \nu\lambda_1( 1+ \log(G)) G^2.
\end{align}
\end{remark}

\bigskip

%----------------------------------------------------------------------
\section{Convergence analysis with observable data of type II}\label{conv_typeII}
%-----------------------------------------------------------------------

Next, we will prove that under certain conditions on $\mu$, the approximate solution $(\uu_1,\uu_2)$ of the data assimilation system \eqref{DA_NSE} exists globally in time and converges to the solution $(u_1,u_2)$ of the 2D Navier-Stokes equations \eqref{2D_NSE_com}, subject to periodic or Dirichlet boundary conditions, respectively, as $t\rightarrow \infty$, when the observable operators satisfy \eqref{app2}. 

\subsection{Dirichlet Boundary Conditions}

The existence and uniqueness of strong solutions of the data assimilation system \eqref{DA_NSE} with observables that satisfy \eqref{app2}, subject to Dirichlet boundary conditions, does not follow immediately as in the case of the 2D Navier-Stokes equations. Extra conditions on the nudging constant $\mu>0$ are required. Next, we will prove the global existence of the strong solution $\uu(t)$ of \eqref{DA_NSE} and simultaneously show that it converges, in time,  to the reference solution $u(t)$ of the 2D Navier-Stokes equations \eqref{2D_NSE_com}. 

\begin{theorem}\label{th_conv_NS_2}
Let $I_h$ satisfy the approximation property \eqref{app2} and $u(t,x,y)$ $=$ $(u_1(t,x,y),u_2(t,x,y))$ be a strong solution in the global attractor of \eqref{2D_NSE_com} subject to Dirichlet boundary conditions.  If $\uu^0 \in V$ with 
\begin{align}\label{initial_condition_nse} 
\norm{\nabla  \uu^0}_{\Lph{2}}^2\leq {\tilde c}\nu^2\lambda_1G^2e^{G^4}, 
\end{align}
where ${\tilde c}$ is the constant in \eqref{dirichlet_H1}, and $\mu>0$ is large enough such that
\begin{align}\label{mu_2D_NSE_2}
\mu \geq 2c \nu\lambda_1K\log(K),
\end{align}
where $K$ is defined in \eqref{J}, and $h>0$ is chosen small enough such that $\mu c_0^2h^2\leq \nu $. Then, there exists a unique strong solution $(\uu_1(t,x,y), \uu_2(t,x,y))$ of \eqref{DA_NSE} that satisfies 
\begin{align*}
\uu \in C([0,T];V) \cap L^2([0,T];\mathcal{D}(A)), \quad \text{and} \quad \od{\uu}{t} \in L^2([0,T];H),  
\end{align*}
such that $$\norm{\nabla  \uu(t)}_{\Lph{2}}^ 2\leq 7{\tilde c}\nu^2\lambda_1G^2e^{G^4},$$ for all $t>0$. Moreover, $\norm{\nabla u(t)-\nabla \uu(t)}_{\Lph{2}}^2$ $\rightarrow 0$, at an 
exponential rate, as $t \rightarrow \infty$.
\end{theorem}
\begin{proof}

We will prove some formal {\it apriori} estimates that are essential in proving the global existence of solutions of system \eqref{DA_NSE}. These estimates can be justified rigorously by using the Galerkin method and the Aubin compactness theorem (see e.g. \cite{Constantin_Foias_1988}).
 
Define $\dU= u-\uu$ and $\dP = p-\p$. Then $\dU_1$ and $\dU_2$ satisfy the equations \eqref{dU_NSE}.  Since by assumption, $(u_1,u_2)$ is a solution which is contained in the global attractor of \eqref{2D_NSE_com}, in particular, it satisfies the global estimates in Proposition \ref{unif_bounds_2D_NSE}, then showing the  global existence, in time, of the solution $(\dU_1, \dU_2)$ is equivalent to showing the global existence, in time, of the solution $(\uu_1, \uu_2)$ of system \eqref{DA_NSE}. To be concise here, we will show the global existence of the solution $\dU(t)$ and show that $\norm{\dU(t)}_{\Lph{2}}^2$ decays exponentially, in time, which will prove the convergence of the approximate solution $\uu(t)$ to the exact solution $u(t)$, exponentially in time.

Since $\norm{\nabla  \uu^0}_{\Lph{2}}^2\leq {\tilde c}\nu^2\lambda_1G^2e^{G^4}$,  then by the continuity of $\norm{\nabla \uu(t)}_{\Lph{2}}^2$, there exists a short time interval $[0,{\tilde{T}})$ such that
\begin{align}\label{V_bound}
\norm{\nabla  \uu(t)}_{\Lph{2}}^2\leq7{\tilde c}\nu^2\lambda_1G^2e^{G^4},
\end{align}
for all $t\in [0,{\tilde{T}})$. Assume $[0, \tilde{T})$ is the maximal finite time interval such that \eqref{V_bound} holds. We will show, by contradiction, that $ \tilde{T} =\infty$. Assume that $\tilde{T}<\infty$, then it is clear that $$\limsup_{t\rightarrow\tilde{T}^{-}} \norm{\nabla \uu(t)}_{\Lph{2}}^2 = 7{\tilde c}\nu^2\lambda_1G^2e^{G^4},$$ otherwise, \eqref{V_bound} holds beyond $\tilde{T}$.
Taking the $\Lph{2}$ inner product of \eqref{dU_1} and \eqref{dU_2} with $-\Delta \dU_1$ and $-\Delta \dU_2$, respectively, we obtain, on the time interval $[0,\tilde{T})$, that 
\begin{align*}
\frac 12 \od{}{t} \norm{\nabla \dU_1}_{\Lph{2}}^2 + \nu \norm{\Delta\dU_1}_{\Lph{2}}^2 &\leq \abs{J_{1a}} + \abs{J_{1b}} + \abs{J_{1c}} + \abs{J_{1b}} + (\pp_x\dP,\Delta\dU_1), \\ 
\frac 12 \od{}{t} \norm{\nabla\dU_2}_{\Lph{2}}^2 + \nu \norm{\Delta \dU_2}_{\Lph{2}}^2 &\leq \abs{J_{2a}} + \abs{J_{2b}} + \abs{J_{2c}} + \abs{J_{2d}} + (\pp_y\dP,\Delta\dU_2) \notag \\
 & \quad - \mu (I_h(\dU_2),-\Delta\dU_2),  
\end{align*}
where
\begin{align*}
J_{1a} := (\uu_1\pp_x\dU_1 ,-\Delta \dU_1), \qquad J_{1b}: = (\uu_2\pp_y\dU_1, -\Delta\dU_1), \\
J_{1c} := (\dU_1\pp_xu_1 ,-\Delta \dU_1), \qquad J_{1d}: = (\dU_2\pp_yu_1, -\Delta\dU_1), \\
J_{2a} : = (\uu_1 \pp_x\dU_2, -\Delta \dU_2), \qquad J_{2b} : = (\uu_2\pp_y\dU_2,-\Delta\dU_2), \\ 
J_{2c} : = (\dU_1 \pp_xu_2, -\Delta \dU_2), \qquad J_{2d} : = (\dU_2\pp_yu_2,-\Delta\dU_2). 
\end{align*}

Thanks to the assumption $\mu c_0^2h^2\leq \nu$, Young's inequality and inequality \eqref{app2}, we have 
\begin{align}\label{1_2}
-\mu(I_h(\dU_2),-\Delta\dU_2) & = -\mu(I_h(\dU_2)-\dU_2, -\Delta\dU_2) - \mu \norm{\nabla\dU_2}_{\Lph{2}}^2\notag \\
& \leq \mu \norm{I_h(\dU_2)-\dU_2}_{\Lph{2}}\norm{\Delta\dU_2}_{\Lph{2}} - \mu \norm{\dU_2}_{\Lph{2}}^2\notag \\
&\leq \frac{\mu^2}{\nu} \norm{I_h(\dU_2)-\dU_2}_{\Lph{2}}^2 + \frac {\nu}{4}\norm{\Delta \dU_2}_{\Lph{2}}^2 - \mu \norm{\nabla\dU_2}_{\Lph{2}}^2\notag \\
& \leq \frac{\mu^2c_0^2h^2}{2\nu} \norm{\nabla \dU_2}_{\Lph{2}}^2 + \frac{\mu^2c_0^4h^4}{4\nu}\norm{\Delta \dU_2}_{\Lph{2}}^2 + \frac {\nu}{4}\norm{\Delta \dU_2}_{\Lph{2}}^2\notag \\ & \qquad - \mu \norm{\nabla\dU_2}_{\Lph{2}}^2 \notag \\ 
& \leq \frac{\nu}{2} \norm{\Delta \dU_2}_{\Lph{2}}^2 - \frac{\mu}{2} \norm{\nabla \dU_2}_{\Lph{2}}^2. 
\end{align}

To estimate the nonlinear terms we proceed as follows: using the divergence free condition \eqref{div_dU} and the logarithmic inequality \eqref{titi_2}: 
\begin{align}\label{2_2}
& J_{1a} = (\uu_1\pp_x\dU_1 ,-\Delta \dU_1) = (\uu_1\pp_y\dU_2 ,\Delta \dU_1)\notag \\
& \leq c_T \norm {\nabla \uu_1}_{\Lph{2}} \norm{\nabla \dU_2}_{\Lph{2}}\norm{\Delta \dU_1}_{\Lph{2}} \left( 1 + \log \left(\frac{\norm{\Delta \dU_2}_{\Lph{2}}}{\lambda_1^{1/2}\norm{\nabla \dU_2}_{\Lph{2}}}\right)\right)^{1/2}\notag \\ 
& \leq \frac{\nu}{100} \norm{\Delta \dU_1}_{\Lph{2}}^2 + \frac{c}{\nu}\norm{\nabla\uu_1}_{\Lph{2}}^2\norm{\nabla \dU_2}_{\Lph{2}}^2 \left( 1 + \log \left(\frac{\norm{\Delta \dU_2}_{\Lph{2}}}{\lambda_1^{1/2}\norm{\nabla \dU_2}_{\Lph{2}}}\right)\right). 
\end{align} 
By a similar argument, we can show 
\begin{align} \label{I2a}
J_{2a} & = (\uu_1 \pp_x\dU_2, -\Delta \dU_2)\notag \\
& \leq \frac{\nu}{100} \norm{\Delta \dU_2}_{\Lph{2}}^2 + \frac{c}{\nu}\norm{\nabla\uu_1}_{\Lph{2}}^2\norm{\nabla \dU_2}_{\Lph{2}}^2 \left( 1 + \log \left(\frac{\norm{\Delta \dU_2}_{\Lph{2}}}{\lambda_1^{1/2}\norm{\nabla \dU_2}_{\Lph{2}}}\right)\right), 
\end{align}
and 
\begin{align}\label{3_2}
J_{2b} & = (\uu_2\pp_y\dU_2,-\Delta\dU_2)\notag \\
& \leq \frac{\nu}{100} \norm{\Delta \dU_2}_{\Lph{2}}^2 + \frac{c}{\nu}\norm{\nabla\uu_2}_{\Lph{2}}^2\norm{\nabla \dU_2}_{\Lph{2}}^2 \left( 1 + \log \left(\frac{\norm{\Delta \dU_2}_{\Lph{2}}}{\lambda_1^{1/2}\norm{\nabla \dU_2}_{\Lph{2}}}\right)\right). 
\end{align}
By the logarithmic inequality \eqref{titi_3}, we have 
\begin{align}\label{4_2} 
J_{1d} &= (\dU_2\pp_yu_1, -\Delta\dU_1)\notag \\ 
& \leq c_T \norm{\nabla\dU_2}_{\Lph{2}}\norm{\nabla u_1}_{\Lph{2}}\norm{\Delta \dU_1}_{\Lph{2}} \left( 1 + \log \left(\frac{\norm{\Delta \dU_2}_{\Lph{2}}}{\lambda_1^{1/2}\norm{\nabla \dU_2}_{\Lph{2}}}\right)\right)^{1/2}\notag \\
& \leq \frac{\nu}{100}\norm{\Delta \dU_1}_{\Lph{2}}^2 + \frac{c}{\nu} \norm{\nabla u_1}_{\Lph{2}}^2 \norm{\nabla \dU_2}_{\Lph{2}}^2 \left( 1 + \log \left(\frac{\norm{\Delta \dU_2}_{\Lph{2}}}{\lambda_1^{1/2}\norm{\nabla \dU_2}_{\Lph{2}}}\right)\right). 
\end{align}
Similarly, we have 
\begin{align}\label{5_2}
J_{2d} &= (\dU_2\pp_yu_2,-\Delta\dU_2)\notag \\
& \leq \frac{\nu}{100}\norm{\Delta \dU_2}_{\Lph{2}}^2 + \frac{c}{\nu} \norm{\nabla u_2}_{\Lph{2}}^2 \norm{\nabla \dU_2}_{\Lph{2}}^2 \left( 1 + \log \left(\frac{\norm{\Delta \dU_2}_{\Lph{2}}}{\lambda_1^{1/2}\norm{\nabla \dU_2}_{\Lph{2}}}\right)\right). 
\end{align}

Integration by parts and the boundary conditions yield  
\begin{align*}
J_{1b} = (\uu_2\pp_y\dU_1, -\Delta\dU_1) & = (\uu_2\pp_y\dU_1, -\pp_{xx}\dU_1) + (\uu_2\pp_y\dU_1, -\pp_{yy}\dU_1) \notag \\
& = (\pp_x\uu_2\pp_y \dU_1, \pp_x\dU_1) + (\uu_2\pp_{yx}\dU_1, \pp_x\dU_1) - (\uu_2\pp_y\dU_1, \pp_{yy}\dU_1)\notag \\
& = : J_{1b1} + J_{1b2} - J_{1b3}. 
\end{align*}
Using  the divergence free condition \eqref{div_dU} and the logarithmic inequality \eqref{titi_1}, we get
\begin{align}\label{6_2}
J_{1b1} &= (\pp_x\uu_2\pp_y \dU_1, \pp_x\dU_1) = -(\pp_y \dU_1\pp_x\uu_2, \pp_y\dU_2) \notag \\
& \leq c_T \norm{\Delta \dU_1}_{\Lph{2}}\norm{\nabla \uu_2}_{\Lph{2}}\norm{\nabla \dU_2}_{\Lph{2}}\left( 1 + \log \left(\frac{\norm{\Delta \dU_2}_{\Lph{2}}}{\lambda_1^{1/2}\norm{\nabla \dU_2}_{\Lph{2}}}\right)\right)^{1/2}\notag \\
& \leq \frac{\nu}{100}\norm{\Delta \dU_1}_{\Lph{2}}^2 + \frac{c}{\nu} \norm{\nabla \uu_2}_{\Lph{2}}^2 \norm{\nabla \dU_2}_{\Lph{2}}^2 \left( 1 + \log \left(\frac{\norm{\Delta \dU_2}_{\Lph{2}}}{\lambda_1^{1/2}\norm{\nabla \dU_2}_{\Lph{2}}}\right)\right). 
\end{align}
Similarly, the divergence free condition \eqref{div_dU} and the logarithmic inequality \eqref{titi_2} imply 
\begin{align}\label{7_2}
J_{1b2} &= (\uu_2\pp_{yx}\dU_1, \pp_x\dU_1)= (\uu_2\pp_{yy}\dU_2, \pp_y \dU_2)= (\uu_2\pp_{y}\dU_2, \pp_{yy} \dU_2) \notag \\
& \leq c_T \norm{\nabla \uu_2}_{\Lph{2}}\norm{\nabla \dU_2}_{\Lph{2}}\norm{\Delta \dU_2}_{\Lph{2}}\left( 1 + \log \left(\frac{\norm{\Delta \dU_2}_{\Lph{2}}}{\lambda_1^{1/2}\norm{\nabla \dU_2}_{\Lph{2}}}\right)\right)^{1/2}\notag \\
& \leq \frac{\nu}{100}\norm{\Delta \dU_2}_{\Lph{2}}^2 + \frac{c}{\nu} \norm{\nabla \uu_2}_{\Lph{2}}^2 \norm{\nabla \dU_2}_{\Lph{2}}^2 \left( 1 + \log \left(\frac{\norm{\Delta \dU_2}_{\Lph{2}}}{\lambda_1^{1/2}\norm{\nabla \dU_2}_{\Lph{2}}}\right)\right). 
\end{align}
By integration by parts and the divergence free condition \eqref{div_uu}, we also have 
\begin{align*}
J_{1b3} &= (\uu_2\pp_y\dU_1, \pp_{yy}\dU_1) = \frac12 (\uu_2\pp_y(\pp_y \dU_1)^2)\notag \\
& = -\frac 12 (\pp_y\uu_2 (\pp_y \dU_1)^2) = \frac 12 (\pp_x\uu_1(\pp_y \dU_1)^2)  \\ 
&= - (\uu_1\pp_y\dU_1, \pp_{xy}\dU_1)  =(\uu_1\pp_y\dU_1 \pp_{yy} \dU_2)\notag \\
& = - (\pp_y\uu_1\pp_y \dU_1, \pp_y\dU_2) - (\uu_1\pp_{yy}\dU_1\pp_y\dU_2) =: J_{1b31} + J_{1b32}.  
\end{align*}
Following similar argument as above, using the logarithmic inequalities \eqref{titi_1} and \eqref{titi_2}, we can show that 
\begin{align}\label{8_2}
J_{1b31} &= -(\pp_y\uu_1\pp_y \dU_1, \pp_y\dU_2)= -(\pp_y \dU_1\pp_y\uu_1, \pp_y\dU_2)
\notag \\
 & \leq \frac{\nu}{100}\norm{\Delta \dU_1}_{\Lph{2}}^2 + \frac{c}{\nu} \norm{\nabla \uu_1}_{\Lph{2}}^2 \norm{\nabla \dU_2}_{\Lph{2}}^2 \left( 1 + \log \left(\frac{\norm{\Delta \dU_2}_{\Lph{2}}}{\lambda_1^{1/2}\norm{\nabla \dU_2}_{\Lph{2}}}\right)\right), 
\end{align}
and 
\begin{align}\label{9_2}
J_{1b32} & = -(\uu_1\pp_{yy}\dU_1,\pp_y\dU_2) = -(\uu_1\pp_y\dU_2,\pp_{yy}\dU_1) \notag \\
& \leq \frac{\nu}{100}\norm{\Delta \dU_1}_{\Lph{2}}^2 + \frac{c}{\nu} \norm{\nabla \uu_1}_{\Lph{2}}^2 \norm{\nabla \dU_2}_{\Lph{2}}^2 \left( 1 + \log \left(\frac{\norm{\Delta \dU_2}_{\Lph{2}}}{\lambda_1^{1/2}\norm{\nabla \dU_2}_{\Lph{2}}}\right)\right). 
\end{align}

Integration by parts, the divergence free condition \eqref{div_dU} and the boundary conditions imply 
\begin{align*}
J_{1c} = (\dU_1\pp_xu_1 ,-\Delta \dU_1) &= (\dU_1\pp_xu_1, -\pp_{xx}\dU_1) + (\dU_1\pp_xu_1, -\pp_{yy}\dU_1) \notag \\
& = -(\pp_{x}\dU_1\pp_xu_1,  \pp_{y}\dU_2) - (\dU_1\pp_{xx}u_1,  \pp_{y}\dU_2) -(\dU_1\pp_xu_1, \pp_{yy}\dU_1)\notag\\
& =: J_{1c1} + J_{1c2} + J_{1c_3}. 
\end{align*}
By the logarithmic inequality \eqref{titi_1}, we have 
\begin{align}
J_{1c1} & = -(\pp_{x}\dU_1\pp_xu_1,  \pp_{y}\dU_2) \notag \\
& \leq  c \norm{\Delta \dU_1}_{\Lph{2}} \norm{\nabla u_1}_{\Lph{2}}\norm{\nabla \dU_2}_{\Lph{2}}  \left( 1 + \log\left(\dfrac{\norm{\Delta \dU_2}}{\lambda_1^{1/2}\norm{\nabla\dU_2}_{\Lph{2}}}\right)\right)^{1/2}\notag \\
& \leq \dfrac{\nu}{100}\norm{\Delta \dU_1}_{\Lph{2}}^2 + \dfrac{c}{\nu} \norm{\nabla u_1}_{\Lph{2}}^2\norm{\nabla \dU_2}_{\Lph{2}}^2\left( 1 + \log\left(\dfrac{\norm{\Delta \dU_2}}{\lambda_1^{1/2}\norm{\nabla\dU_2}_{\Lph{2}}}\right)\right). 
\end{align}
Applying the logarithmic inequality \eqref{titi_2} and the Poincar\'e inequality yield
\begin{align}
J_{1c2}  &= - (\dU_1\pp_{xx}u_1,  \pp_{y}\dU_2)= -(\dU_1\pp_y\dU_2, \pp_{xx}u_1)\notag \\
&\leq c \norm{\Delta u_1}_{\Lph{2}} \norm{\nabla \dU_1}_{\Lph{2}}\norm{\nabla \dU_2}_{\Lph{2}}   \left( 1 + \log\left(\dfrac{\norm{\Delta \dU_2}}{\lambda^{1/2}\norm{\nabla\dU_2}_{\Lph{2}}}\right)\right)^{1/2}\notag \\
& \leq \dfrac{\nu}{100}\norm{\Delta \dU_1}_{\Lph{2}}^2 +\dfrac{c\,\lambda_1^{-1}  }{\nu} \norm{\Delta u_1}_{\Lph{2}}^2\norm{\nabla \dU_2}_{\Lph{2}}^2\left( 1 + \log\left(\dfrac{\norm{\Delta \dU_2}}{\lambda^{1/2}\norm{\nabla\dU_2}_{\Lph{2}}}\right)\right).
\end{align}
The term $J_{1c_3}$ can be estimated by a similar argument as above. Also, using integration by parts and the divergence free condition \eqref{div_u}, we get
\begin{align*}
J_{2c} : = (\dU_1 \pp_xu_2, -\Delta \dU_2) &=  (\dU_1 \pp_xu_2, -\pp_{xx}\dU_2) + (\dU_1 \pp_xu_2, -\pp_{yy} \dU_2)  \\
&=  (\pp_x\dU_1 \pp_xu_2, \pp_{x}\dU_2) + (\dU_1 \pp_{xx}u_2, \pp_{x}\dU_2)  \\
&\qquad +(\pp_y\dU_1 \pp_xu_2, \pp_{y} \dU_2) + (\dU_1 \pp_{xy}u_2, \pp_{y} \dU_2)\\
&=  (\pp_x\dU_1 \pp_xu_2, \pp_{x}\dU_2) + (\dU_1 \pp_{xx}u_2, \pp_{x}\dU_2)  \\
&\qquad +(\pp_y\dU_1 \pp_xu_2, \pp_{y} \dU_2) - (\dU_1 \pp_{xx}u_1, \pp_{y} \dU_2)\\
& =: J_{2c1}+ J_{2c2} + J_{2c3} + J_{2c4}. 
\end{align*}
The above four  terms can be also estimated using the logarithmic inequality \eqref{titi_1} as shown previously. 

We note that
\begin{align}
(\pp_x\dP, \Delta\dU_1) + (\pp_y\dP, \Delta\dU_2) = 0, 
\end{align}
due to the divergence free condition \eqref{div_dU} and integration by parts.   It follows from the above estimates that on the time interval $[0,\tilde{T})$, we have 
\begin{align*}
& \od{}{t} \norm{\nabla \dU}_{\Lph{2}}^2 + \frac{\nu}{2}\norm{\Delta \dU}_{\Lph{2}}^2 \leq \notag \\
&\quad \frac{c}{\nu} \left(\norm{\nabla u}_{\Lph{2}}^2 + \norm{\nabla \uu}_{\Lph{2}}^2 + \lambda_1^{-1} \norm{\Delta u})_{\Lph{2}}^2\right) \left( 1 + \log \left(\frac{\norm{\Delta \dU_2}_{\Lph{2}}^2}{\lambda_1\norm{\nabla \dU_2}_{\Lph{2}}^2}\right)\right)\norm{\nabla \dU_2}_{\Lph{2}}^2 \notag \\ 
& \quad  - \mu \norm{\nabla \dU_2}_{\Lph{2}}^2. 
\end{align*}
Using the Poincar\'e inequality, we conclude that 
\begin{align}\label{poincare_ineq_nse}
& \od{}{t} \norm{\nabla \dU}_{\Lph{2}}^2 + \frac{\nu\lambda_1}{2}\norm{\nabla\dU_1}_{\Lph{2}}^2  + \beta(t)\norm{\nabla \dU_2}_{\Lph{2}}^2 \leq 0, 
\end{align}
with 
$$ \beta(t) = \frac{\nu\lambda_1}{2}\left[r - \gamma(t) (1+ \log(r))\right] + \mu,$$ 
where we denoted by 
$$ r = \frac{\norm{\Delta \dU_2}_{\Lph{2}}^2}{\lambda_1\norm{\nabla \dU_2}_{\Lph{2}}^2}, $$ 
and 
$$ \gamma (t) = \frac{c}{\nu^2\lambda_1}\left(\norm{\nabla u(t)}_{\Lph{2}}^2 + \norm{\nabla \uu(t)}_{\Lph{2}}^2 + \lambda_1^{-1} \norm{\Delta u(t)})_{\Lph{2}}^2\right). $$ 

On the time interval $[0,\tilde{T})$, we have $\norm{\nabla\uu}_{\Lph{2}}^2 \leq 5{\tilde c}\nu^2\lambda_1G^2e^{G^4}$. Thus, by \eqref{dirichlet_H1} and \eqref{dirichlet_H2}, we have 
\begin{align}\label{J}
\gamma(t) &\leq \frac{c}{\nu^2 \lambda_1} \left(8{\tilde c}\nu^2\lambda_1G^2e^{G^4}+ c\nu^2\lambda_1\left(1+\left(1+ G^2e^{G^4}\right)\left(1+e^{G^4}+G^4e^{G^4}\right)\right)\right)\notag \\
& \leq c G^2\left(1+\left(1+ G^2e^{G^4}\right)\left(1+e^{G^4}+G^4e^{G^4}\right)\right)=: K, 
\end{align}
for all $t\in [0,\tilde{T})$. 
By Lemma \ref{log_prop}, we may conclude that 
\begin{align}
\beta(t) &\geq -\nu \lambda_1 \left(\gamma(t) \log \left(\gamma(t))\right)\right) + \mu \notag \\
& \geq \mu - c \nu\lambda_1 K\log(K). 
\end{align}

Therefore, the assumption \eqref{mu_2D_NSE_1} implies that $\beta(t)>c \nu\lambda_1 K\log(K)>0$ for all $t \in [0,\tilde{T})$. Define $\tilde{\beta}(t) := \min\{\frac{\nu\lambda_1}{2}, \,\beta(t)\} >0$, then we can write \eqref{poincare_ineq_nse} as 
$$ \od{}{t} \norm{\nabla \dU}_{\Lph{2}}^2+ \tilde{\beta}(t)\norm{\nabla \dU}_{\Lph{2}}^2 \leq 0 .$$  Using Gronwall's Lemma, we have 
\begin{align}\label{exponential_decay} 
\norm{\nabla \dU(t)}_{\Lph{2}}^2 \leq \norm{\nabla \dU(0)}_{\Lph{2}}^2 e^{\max\{-\nu\lambda_1, -c \nu\lambda_1 K\log(K)t},
\end{align}
for all $t \in [0,\tilde{T})$. Since $\norm{\nabla \dU(0)}_{\Lph{2}}^2 \leq 2\norm{\nabla u(0)}_{\Lph{2}}^2 + 2\norm{\nabla \uu(0)}_{\Lph{2}}^2$, then by \eqref{jolly} and \eqref{initial_condition_nse}, and due to Remark \ref{t0_remark}, we have 
\begin{align}\label{exp_decay_nse}
\norm{\nabla \dU(t)}_{\Lph{2}}^2\leq  \norm{\nabla\dU(0)}_{\Lph{2}}^2 \leq 4{\tilde c} \nu^2\lambda_1G^2e^{G^4}, 
\end{align}
for all $t \in [0,\tilde{T})$. This implies that $\norm{\nabla \uu(t)}_{\Lph{2}}^2 \leq 6 \tilde{c}\nu^2\lambda_1 G^2e^{G^4}$, for all $t\in [0,\tilde{T})$. This in turn will yield a contradiction since
$$ 7\tilde{c}\nu^2\lambda_1 G^2e^{G^4} = \limsup_{t\rightarrow \tilde{T}^{-}} \norm{\nabla \uu(t)}_{\Lph{2}}^2 \leq 6 \tilde{c}\nu^2\lambda_1 G^2e^{G^4}.$$
This proves that $\tilde{T} = \infty$. Thus, the solution $\dU(t)$ exists globally in time and it satisfies 
\begin{align}\label{exp_decay_nse_2}
\norm{\nabla \dU(t)}_{\Lph{2}}^2 \leq 4\tilde{c}\nu^2\lambda_1 G^2e^{G^4}, 
\end{align}
for all $t\geq 0$.
Following the techniques that were introduced to prove the existence and uniqueness of solutions for the Navier-Stokes equations (see for example,  \cite{Constantin_Foias_1988}, \cite{Temam_1997} and \cite{Temam_2001_Th_Num}), we can show the existence of the solution $ \dU(t)$ of system \eqref{dU_NSE}, that will inherit the estimate \eqref{exp_decay_nse}. Moreover, $\norm{\dU(t)}_{\Lph{2}}^2$ decays exponentially in time and it inherits the inequality \eqref{exponential_decay}. The uniqueness and the well-posedness will follow by a similar argument as above. Since, by assumption, $u(t)$ is a strong solution in the global attractor of the 2D Navier-Stokes equations, then this proves the global existence and the uniqueness of the solution $\uu(t)$ of system \eqref{DA_NSE}, that satisfies $\norm{\nabla \uu(t)}_{\Lph{2}}^2 \leq 7 \tilde{c}\nu^2\lambda_1 G^2e^{G^4}$, for all $t\geq 0$. Moreover, we have  
$$\norm{\nabla u(t) - \nabla \uu(t)}_{\Lph{2}}^2 \rightarrow 0,$$ at an exponential rate as $t\rightarrow\infty$.
\end{proof}
\bigskip 

\subsection{Periodic Boundary Conditions} 

The existence and uniqueness of strong solutions of the data assimilation algorithm \eqref{DA_NSE} with observables that satisfy \eqref{app2}, subject to periodic conditions, as stated below, follows by a similar argument as for the two-dimensional Navier-Stokes equations. See \cite{Azouani_Olson_Titi}, \cite{Constantin_Foias_1988}, \cite{Robinson} and \cite{Temam_1997} for more details. 
\begin{theorem}
Suppose $I_h$ satisfy \eqref{app2} and $\mu>0$ and $h>0$ are chosen such that $\mu c_0^2h^2\leq \nu$, where $c_0$ is the constant in \eqref{app}. If the initial data $\uu_0 \in V$, then the continuous data assimilation system \eqref{DA_NSE}, subject to periodic boundary conditions, possess a unique global strong solution $\uu(t, x,y)= (\uu_1(t, x,y), \uu_2(t , x,y))$ that satisfies 
\begin{align*}
\uu \in C([0,T];V) \cap L^2([0,T];\mathcal{D}(A)), \quad \text{and} \quad \od{\uu}{t} \in L^2([0,T];H). 
\end{align*} 
Moreover, the solution $\uu(t, x,y)$ depends continuously on the initial data $\uu_0$. 
\end{theorem}

\begin{theorem}\label{th_conv_NS_3}
Suppose that $I_h$ satisfy the approximation property \eqref{app2} and $u(t,x,y)$ $=$ $(u_1(t,x,y),u_2(t,x,y))$ is a strong solution in the global attractor of \eqref{2D_NSE_com} subject to periodic boundary conditions. Let $\uu(t,x,y)$ $=$ $(\uu_1(t,x,y),\uu_2(t,x,y))$ be a strong solution of \eqref{DA_NSE}, subject to periodic boundary conditions. If $\mu>0$ is chosen large enough such that 
\begin{align}\label{mu_2D_NSE_per}
\mu \geq 2 c\nu\lambda_1(G^2+ G^3), 
\end{align}
and $h>0$ is chosen small enough such that $\mu c_0^2h^2\leq \nu $, then, $\norm{\nabla u(t)-\nabla \uu(t)}_{\Lph{2}}^2$ $\rightarrow 0$, at an 
exponential rate, as $t \rightarrow \infty$.
\end{theorem}

\begin{proof} 
The proof follows by similar means as in \cite{Azouani_Olson_Titi}. Taking the inner product of \eqref{dU_1} and \eqref{dU_2} with $-\Delta \dU_1$ and $-\Delta \dU_2$, respectively, adding the two equations and then using the orthogonality property \eqref{per_orth_2}, we get
\begin{align*}
\frac{1}{2} \od{}{t}\norm{\nabla \dU}_{\Lph{2}}^2 + \nu \norm{\Delta \dU}_{\Lph{2}}^2 \leq \abs{\left((\dU\cdot\nabla)\dU, \Delta u\right)} - \mu(I_h(\dU_2), -\Delta \dU_2), 
\end{align*}
where $u$ is the reference solution of the 2D Navier-Stokes equations \eqref{2D_NSE_com}. 

Notice that, thanks to the divergence free condition \eqref{div_dU} and the boundary conditions, we have 
\begin{align*}
\left((\dU\cdot\nabla)\dU, \Delta u\right) &= \left(\dU_1\pp_x\dU, \Delta u\right) +  \left(\dU_2\pp_y\dU, \Delta u\right)\\
& = \left(\dU_1\pp_x\dU_1, \Delta u_1\right) + \left(\dU_1\pp_x\dU_2, \Delta u_2\right) + \left(\dU_2\pp_y\dU, \Delta u\right)\\
& = -\left(\dU_1\pp_y\dU_2, \Delta u_1\right) + \left(\dU_1\pp_x\dU_2, \Delta u_2\right) + \left(\dU_2\pp_y\dU, \Delta u\right)\\
& =: J_{1a} + J_{1b} + J_2. 
\end{align*}

Using the logarithmic inequality \eqref{titi_2} and the Poincar\'e inequality \eqref{poincare}, it follows that 
\begin{align*}
\abs{J_{1a}} &= \abs{\left(\dU_1\pp_y\dU_2, \Delta u_1\right)} \notag \\
& \leq c_T\norm{\nabla \dU_1}_{\Lph{2}} \norm{\nabla \dU_2}_{\Lph{2}}  \norm{\Delta u_1}_{\Lph{2}}\left(1+ \log\left(\frac{\norm{\Delta \dU_2}_{\Lph{2}}}{\lambda_1^{1/2}\norm{\nabla \dU_2}_{\Lph{2}}}\right)\right)^{1/2}. \notag \\
& \leq c_T\lambda_1^{-1/2}\norm{\Delta \dU_1}_{\Lph{2}} \norm{\nabla \dU_2}_{\Lph{2}}  \norm{\Delta u_1}_{\Lph{2}}\left(1+ \log\left(\frac{\norm{\Delta \dU_2}_{\Lph{2}}}{\lambda_1^{1/2}\norm{\nabla \dU_2}_{\Lph{2}}}\right)\right)^{1/2}. 
\end{align*}
By Young's inequality we get 
\begin{align}\label{11_per}
\abs{J_{1a}} \leq \frac{\nu}{4}\norm{\Delta \dU_1}_{\Lph{2}}^2 + \frac{c}{\nu\lambda_1}  \norm{\Delta u_1}_{\Lph{2}}^2\norm{\nabla \dU_2}_{\Lph{2}}^2 \left(1+ \log\left(\frac{\norm{\Delta \dU_2}_{\Lph{2}}^2}{\lambda_1\norm{\nabla \dU_2}_{\Lph{2}}^2}\right)\right),
\end{align}
for some positive dimensionless constant $c$. Similarly, by \eqref{titi_2}, we can show that 
\begin{align}
\abs{J_{1b}} \leq \frac{\nu}{4}\norm{\Delta \dU_1}_{\Lph{2}}^2 + \frac{c}{\nu\lambda_1}  \norm{\Delta u_2}_{\Lph{2}}^2\norm{\nabla \dU_2}_{\Lph{2}}^2 \left(1+ \log\left(\frac{\norm{\Delta \dU_2}_{\Lph{2}}^2}{\lambda_1\norm{\nabla \dU_2}_{\Lph{2}}^2}\right)\right) \label{22_per}, 
\end{align}
and by \eqref{titi_3}, we have 
\begin{align}
\abs{J_{2}} \leq \frac{\nu}{4}\norm{\Delta \dU}_{\Lph{2}}^2 + \frac{c}{\nu\lambda_1}  \norm{\Delta u}_{\Lph{2}}^2\norm{\nabla \dU_2}_{\Lph{2}}^2 \left(1+ \log\left(\frac{\norm{\Delta \dU_2}_{\Lph{2}}^2}{\lambda_1\norm{\nabla \dU_2}_{\Lph{2}}^2}\right)\right), 
\end{align}
for some positive dimensionless constant $c$. By \eqref{1_2}, we also have
\begin{align}\label{44_per}
-\mu(I_h(\dU_2^n),-\Delta\dU_2^n) &\leq \frac{\nu}{2} \norm{\Delta \dU_2^n}_{\Lph{2}}^2 - \frac{\mu}{2} \norm{\nabla \dU_2^n}_{\Lph{2}}^2. 
\end{align}

Therefore, by the estimates \eqref{11_per}--\eqref{44_per}, we obtain 
\begin{align*}
& \od{}{t} \norm{\nabla\dU}_{\Lph{2}}^2 + \frac\nu2\norm{\Delta \dU}_{\Lph{2}}^2 \leq \notag \\
& \qquad \qquad \left(\frac{ c}{\nu}\norm{\Delta u}_{\Lph{2}}^2\left(1 + \log \left(\frac{\norm{\Delta \dU_2}_{\Lph{2}}}{\lambda_1^{1/2} \norm{\nabla\dU_2}_{\Lph{2}}}\right)\right)-\mu\right)\norm{\nabla\dU_2}_{\Lph{2}}^2. 
\end{align*}
Using the Poincar\'e inequality \eqref{poincare}, we may rewrite the above inequality as 
\begin{align}\label{conv_est_1_per}
& \od{}{t} \norm{\nabla\dU}_{\Lph{2}}^2 + \frac{\nu\lambda_1}{4}\norm{\nabla\dU}_{\Lph{2}}^2 +  \frac{\nu\lambda_1}{4}\frac{\norm{\Delta \dU_2}_{\Lph{2}}^2}{\lambda_1\norm{\nabla\dU_2}_{\Lph{2}}^2}\norm{\nabla\dU_2}_{\Lph{2}}^2\leq \notag \\  
& \qquad \qquad \left(\frac{ c}{\nu\lambda_1}\norm{\Delta u}_{\Lph{2}}^2\left(1 + \log \left(\frac{\norm{\Delta \dU_2}_{\Lph{2}}^2}{\lambda_1\norm{\nabla\dU_2}_{\Lph{2}}^2}\right)\right)-\mu\right)\norm{\nabla\dU_2}_{\Lph{2}}^2,  
\end{align}
or
\begin{align}\label{conv_est_2_per}
& \od{}{t} \norm{\nabla\dU}_{\Lph{2}}^2 + \frac{\nu\lambda_1}{4}\norm{\Delta\dU}_{\Lph{2}}^2 +  \frac{\nu\lambda_1}{4}\left(\phi(r(t)) +\frac{4\mu}{\nu\lambda_1}\right)\norm{\nabla\dU_2}_{\Lph{2}}^2 \leq 0, 
\end{align}
where we denoted by 
$$\phi(r(t)) := r(t) - \gamma(t)(1+\log(r(t)),$$
$$ r(t): = \frac{\norm{\Delta \dU_2}_{\Lph{2}}^2}{\lambda_1 \norm{\nabla\dU_2}_{\Lph{2}}^2}, \quad \gamma(t) : = \frac{c}{\nu^2\lambda_1^2}\norm{\Delta u}_{\Lph{2}}^2. $$
Now, Lemma \ref{log_prop} implies that
$$\phi(r(t)) \geq - \gamma \log(\gamma) = - \frac{c}{\nu^2\lambda_1^2}\norm{\Delta u}_{\Lph{2}}^2\log\left(\frac{c}{\nu^2\lambda_1^2}\norm{\Delta u}_{\Lph{2}}^2\right).$$
Setting 
$$\beta(t) := \mu -  \frac{c}{\nu\lambda_1}\norm{\Delta u}_{\Lph{2}}^2\log\left(\frac{c}{\nu^2\lambda_1^2}\norm{\Delta u}_{\Lph{2}}^2\right), $$
we have 
\begin{align}
& \od{}{t} \norm{\nabla\dU}_{\Lph{2}}^2 + \frac{\nu\lambda_1}{4}\norm{\nabla\dU}_{\Lph{2}}^2  + \beta(t) \norm{\nabla\dU_2}_{\Lph{2}}^2 \leq 0. 
\end{align}

We may conclude that 
\begin{align*}
\od{}{t} \norm{\nabla\dU}_{\Lph{2}}^2 +\min\{\frac{\nu\lambda_1}{4}, \,\beta(t)\}\norm{\nabla\dU}^2_{\Lph{2}}\leq 0. 
\end{align*}
Taking $\tau = (\nu\lambda_1)^{-1}$ in Proposition \ref{unif_bounds_2D_NSE}, using \eqref{per_est} and \eqref{jolly}, and due to Remark \ref{t0_remark}, we conclude that
\begin{align}\label{major_est_1}
\frac{c}{\nu\lambda_1}\int_t^{t+\tau}&\norm{\Delta u(s)}_{\Lph{2}}^2\log\left(\frac{c}{\nu^2\lambda_1^2}\norm{\Delta u(s)}_{\Lph{2}}^2\right)\,ds \notag\\
  &\leq  \frac{c}{\nu\lambda_1}(1 + G) \int_t^{t+\tau} \norm{\Delta u(s)}_{\Lph{2}}^2\,ds \notag \\
 & \leq c(1 + G) G^2 = c(G^2+ G^3), 
\end{align} 
for all $t\geq 0$. Therefore, the assumption \eqref{mu_2D_NSE_per} implies that 
\begin{align}
\liminf_{t\rightarrow\infty} \int_t^{t+\tau} \beta(s)\,ds \geq \frac{\mu\tau}{2}>0, \quad \text{and} \quad \limsup_{t\rightarrow\infty} \int_t^{t+\tau} \beta(s)\,ds \leq\frac{3\mu\tau}{2} < \infty. 
\end{align}
Define $\tilde{\beta}(t) := \min\{\frac{\nu\lambda_1}{4}, \,\beta(t)\}$, then $\tilde{\beta}(t)$ satisfies the condition \eqref{uniform_conditions}. By the uniform Gronwall's  lemma \eqref{gen_gron_2}, we obtain 
$$\norm{\nabla \dU}_{\Lph{2}}^2= \norm{\nabla u(t) - \nabla \uu(t)}_{\Lph{2}}^2 \rightarrow 0,$$ at an exponential rate as $t\rightarrow\infty$.
\end{proof}

\bigskip
%-----------------------------------------------------------------------
\section{Appendix A}
%-----------------------------------------------------------------------
In this appendix, we will prove estimate \eqref{dirichlet_H2} following the calculations in \cite{Robinson}. The proof is formal. It can be done rigorously using the Galerkin approximation and then passing to the limit.

It follows from the 2D Navier-Stokes equations \eqref{2D_NSE_com} that 
\begin{align*}
\norm{\pd{u}{t}}_{\Lph{2}} \leq \nu \norm{\Delta u}_{\Lph{2}} + \norm{(u\cdot\nabla)u}_{\Lph{2}} + \norm{f}_{\Lph{2}}. 
\end{align*}
Using the Ladyzhenskaya inequality \eqref{L4_to_H1}, we can show  that if $u\in D(A)$ then 
\begin{align}\label{non_est}
\norm{(u\cdot\nabla)u}_{\Lph{2}} \leq c_1 \norm{u}_{\Lph{2}}^{1/2}\norm{\nabla u}_{\Lph{2}}\norm{\Delta u}_{\Lph{2}}^{1/2}. 
\end{align}  
Using Young's inequality, we have 
\begin{align}\label{inequality}
\norm{\pd{u}{t}}_{\Lph{2}} &\leq \frac{3\nu}{2} \norm{\Delta u}_{\Lph{2}} + \frac{c_1^2}{\nu}\norm{u}_{\Lph{2}}\norm{\nabla u}_{\Lph{2}}^2+ \norm{f}_{\Lph{2}}. 
\end{align}
The estimates \eqref{dirichlet_L2}, \eqref{dirichlet_H1} and the definition of the Grashof number \eqref{Grashof_2} imply that there exists a time $t_0>0$ such that for $t\geq t_0$
\begin{align}
\norm{\pd{u}{t}}_{\Lph{2}}^2 \leq c\nu^2 \norm{\Delta u}_{\Lph{2}}^2 + cG^2 \left(\nu^2\lambda_1G^2e^{G^4}\right)^2 + c \nu^4\lambda_1^2G^2, 
\end{align}
for some positive non-dimensional constant $c$. 

Moreover, integrating inequality \eqref{inequality} on the time interval $[t+\tau]$ for some $\tau>0$ and using the definition of the Grashof number \eqref{Grashof_2} yield
\begin{align*}
&\int_t^{t+\tau}\norm{\pd{u}{t}(s)}_{\Lph{2}}^2\,ds \notag \\
&\leq c\left(\nu^2 \int_{t}^{t+\tau}\norm{\Delta u(s)}_{\Lph{2}}^2\,ds + \frac{1}{\nu}\int_t^{t+\tau}\norm{u}_{\Lph{2}}^2\norm{\nabla u}_{\Lph{2}}^4\, ds + \tau \nu^4\lambda_1^2G^2\right),
\end{align*}
for all $t\geq t_0$. 

Taking $\tau = (\nu\lambda_1)^{-1}$ in Proposition \ref{unif_bounds_2D_NSE} and using the estimates \eqref{dirichlet_L2}, \eqref{dirichlet_H1} and \eqref{dirichlet_int_H2}, we get
\begin{align}\label{time_derivative_int}
&\int_{t}^{t+\tau}\norm{\pd{u}{t}(s)}_{\Lph{2}}^2\,ds \notag\\
 &\leq c\left(\nu^2 \int_{t}^{t+\tau}\norm{\Delta u(s)}_{\Lph{2}}^2\,ds + \nu G^2\int_t^{t+\tau}\norm{\nabla u}_{\Lph{2}}^4\, ds + \tau \nu^4\lambda_1^2G^2\right)\notag \\
& \leq c\left(\nu^2 \int_{t}^{t+\tau}\norm{\Delta u(s)}_{\Lph{2}}^2\,ds + \nu^2\lambda_1G^4e^{G^4}\int_t^{t+\tau}\norm{\nabla u}_{\Lph{2}}^2\, ds + \tau \nu^4\lambda_1^2G^2\right)\notag\\
& \leq c\left(\nu^3\lambda_1G^2(e^{G^4}+\tau\nu\lambda_1) + \nu^3\lambda_1(1+\tau\nu\lambda_1)G^6e^{G^4}+ \tau \nu^4\lambda_1^2G^2\right)\notag\\
& \leq c\nu^3\lambda_1G^2\left((1+e^{G^4} + G^4 e^{G^4}\right).
\end{align}

We now differentiate the equations \eqref{2D_NSE_com} w.r.t. to $t$ and take the inner product with $\pd{u}{t}$ to obtain 
\begin{align*}
\frac{1}{2} \od{}{t} \norm{\pd{u}{t}}_{\Lph{2}}^2 + \nu\norm{\nabla \pd{u}{t}}_{\Lph{2}}^2 &\leq \abs{\left(\left(\pd{u}{t} \cdot\nabla\right)u, \pd{u}{t}\right)}\\
&\leq c_L\norm{\nabla u}_{\Lph{2}}\norm{\pd{u}{t}}_{\Lph{2}}\norm{\nabla\pd{u}{t}}_{\Lph{2}}\notag \\
&\leq \frac{\nu}{2} \norm{\nabla\pd{u}{t}}_{\Lph{2}}^2 + \frac{c_L^2}{2\nu}\norm{\nabla u}_{\Lph{2}}^2\norm{\pd{u}{t}}_{\Lph{2}}^2, 
\end{align*}
where we used the Ladyzhenskaya inequality \eqref{L4_to_H1} and Young's inequality in the last two steps. Integrating the inequality on the time interval $[s, t+\tau]$ and using \eqref{dirichlet_H1} we get 
\begin{align*}
\norm{\pd{u}{t}(t+\tau)}_{\Lph{2}}^2 &\leq \norm{\pd{u}{t}(s)}_{\Lph{2}}^2 + c\nu \lambda_1G^2e^{G^4} \int_{s}^{t+\tau}\norm{\pd{u}{t}(l)}_{\Lph{2}}^2\ dl,
\end{align*}
for all $t_0\leq t\leq s\leq t+\tau$. Integrating once again w.r.t. to $s$ on the time interval $[t, t+\tau]$ and using \eqref{time_derivative_int} yield 
\begin{align*}
\tau \norm{\pd{u}{t}(t+\tau)}_{\Lph{2}}^2 &\leq \left(1+ c\nu\lambda_1\tau G^2e^{G^4}\right)\int_{t}^{t+\tau}\norm{\pd{u}{t}(s)}_{\Lph{2}}^2\ ds\notag \\
& \leq  c\nu^3\lambda_1G^2 \left(1+ c\nu\lambda_1\tau G^2e^{G^4}\right)\left(1+e^{G^4} + G^4e^{G^4}\right).
\end{align*}
Since $\tau=(\nu\lambda_1)^{-1}$, then
\begin{align*}
\norm{\pd{u}{t}(t+\tau)}_{\Lph{2}}^2 \leq c\nu^4\lambda_1^2G^2\left(1+G^2e^{G^4}\right)\left(1+e^{G^4} + G^4e^{G^4}\right),
\end{align*}
for all $t\geq t_0$. We can redefine $t_0$ to be large enough such that for all $t\geq t_0$
\begin{align}\label{time_derivative}
\norm{\pd{u}{t}(t)}_{\Lph{2}}^2 \leq c\nu^4\lambda_1^2G^2\left(1+G^2e^{G^4}\right)\left(1+e^{G^4} + G^4e^{G^4}\right). 
\end{align}

The 2D Navier-Stokes equations \eqref{2D_NSE_com} imply that 
\begin{align*}
\nu\norm{\Delta u}_{\Lph{2}}\leq \norm{\pd{u}{t}}_{\Lph{2}} + \norm{(u\cdot\nabla)u}_{\Lph{2}} + \norm{f}_{\Lph{2}}. 
\end{align*}
This, using the estimate \eqref{time_derivative}, the nonlinearity estimate \eqref{non_est}, the estimates \eqref{dirichlet_L2} and \eqref{dirichlet_H1} and the definition of the Grashof number \eqref{Grashof_2}, we can conclude that for $t\geq t_0$
\begin{align}\label{e_finale}
\norm{\Delta u(t)}_{\Lph{2}}^2 \leq {\tilde c}\nu^2\lambda_1^2G^2\left(1+\left(1+G^2e^{G^4}\right)\left(1+e^{G^4} + G^4e^{G^4}\right)\right),
\end{align}
for some positive non-dimensional constant ${\tilde c}$.

 \begin{remark} 
One can also follow the argument in \cite{Constantin_Foias_1988} using the analyticity of the solution $u(t)$, in time, and the Cauchy integral formula to derive an estimate on $\norm{\pd{u}{t}}_{\Lph{2}}$ . The estimate on $\norm{\Delta u}_{\Lph{2}}$ will then follow by a similar argument as presented above and bounded by a similar estimate as in \eqref{e_finale}.
\end{remark} 

\bigskip
%-----------------------------------------------------------------------
\section*{Acknowledgements}
The work of A.F. is supported in part by NSF grant  DMS-1418911. The work of E.L. is supported  by the ONR grant N001614WX30023. The work of  E.S.T.  is supported in part by a grant of the ONR and the NSF grants DMS-1109640 and DMS-1109645.

\bigskip
%---------------------------------------------------------------

%---------------------------------------------------------------------
\end{document}